\documentclass[10pt, a4paper, reqno, english]{amsart}

\usepackage{amsfonts, amsthm, amsmath, amssymb, mathrsfs}

\usepackage{dsfont}
\usepackage{mathtools}

\usepackage{enumitem}
\usepackage[all]{xy} 
\usepackage{hyperref}

\theoremstyle{definition}
\newtheorem{defin}{Definition}
\newtheorem{example}[defin]{Example}

\theoremstyle{plain}
\newtheorem{theorem}[defin]{Theorem}
\newtheorem{prop}[defin]{Proposition}
\newtheorem{lemma}[defin]{Lemma}
\newtheorem{construction}[defin]{Construction}

\theoremstyle{remark}
\newtheorem{remark}[defin]{Remark}

\numberwithin{defin}{section}
\numberwithin{equation}{section}

\newcommand{\OO}{\mathcal{O}}
\newcommand{\GG}{\mathds{G}}
\newcommand\Gm{\GG_\mathrm{m}}

\newcommand\GmKbar{\GG_{\mathrm{m},\Kbar}}
\newcommand\GmZ{\GG_{\mathrm{m},\ZZ}}
\newcommand\GmO{\GG_{\mathrm{m},\OO_K}}
\newcommand{\AAA}{\mathds{A}} 
\newcommand{\xx}{\mathbf{x}}

\DeclareMathOperator{\Pic}{Pic}
\DeclareMathOperator{\Spec}{Spec}
\DeclareMathOperator{\vol}{vol}
\DeclareMathOperator{\Hom}{Hom}
\DeclareMathOperator{\Gal}{Gal}
\DeclareMathOperator{\Res}{Res}
\DeclareMathOperator{\Tr}{Tr}
\DeclareMathOperator{\Fr}{Fr}

\newcommand{\Aone}{{\mathbf A}_1}
\newcommand{\Athree}{{\mathbf A}_3}
\newcommand{\Dfour}{{\mathbf D}_4}
\newcommand{\tS}{{\widetilde S}}
\newcommand{\tSS}{\widetilde{\mathcal{S}}}
\newcommand{\tX}{{\widetilde X}}
\newcommand\dd{\,\mathrm{d}}

\DeclareMathOperator{\rk}{rk}
\newcommand\Kbar{{\overline{K}}}
\newcommand{\ex}[1]{*+<5pt>[o][F]{#1}}
\newcommand\vv{\mathbf{v}}
\newcommand\ZZp{\ZZ_{>0}}
\newcommand\ZZnn{\ZZ_{\ge 0}}
\newcommand\RRnn{\RR_{\ge 0}}
\newcommand\RRnz{\RR_{\ne 0}}
\newcommand{\congr}[3]{{#1} \equiv {#2} \pmod{#3}}
\newcommand{\congrfr}[3]{{#1} \equiv_{#3} {#2}}
\DeclareMathOperator\supp{supp}
\newcommand\floor[1]{\left\lfloor{#1}\right\rfloor}
\newcommand\ceil[1]{\left\lceil{#1}\right\rceil}
\newcommand\sums[1]{\sum_{\substack{#1}}}
\newcommand\prods[1]{\prod_{\substack{#1}}}
\newcommand\ints[1]{\int_{\substack{#1}}}
\newcommand\legendre[2]{\left(\frac{#1}{#2}\right)}
\newcommand\hilbert[3]{\legendre{#1,#2}{#3}}
\newcommand\bigwhere[2]{\left\{#1\ \left|\ \begin{aligned}#2\end{aligned}\right.\right\}}
\newcommand{\mon}{A}

\newcommand{\Is}{\mathscr{I}}
\newcommand\ab{\underline{a}}
\newcommand\cfrb{{\underline{\mathfrak{c}}}}
\newcommand{\cfr}{\mathfrak{c}}
\newcommand\dfrb{{\underline{\mathfrak{d}}}}
\newcommand\dfr{{\mathfrak{d}}}
\newcommand\bfr{{\mathfrak{b}}}
\newcommand\Bfr{{\mathfrak{B}}}
\newcommand\Afr{{\mathfrak{A}}}
\newcommand\afr{{\mathfrak{a}}}
\newcommand\ffr{{\mathfrak{f}}}
\newcommand\afrb{\underline{\afr}}
\newcommand\p{{\mathfrak{p}}}
\newcommand\gfr{{\mathfrak{g}}}
\newcommand\qfr{{\mathfrak{q}}}
\newcommand{\CC}{\mathds{C}}
\newcommand{\QQ}{\mathds{Q}}
\newcommand{\ZZ}{\mathds{Z}}
\newcommand{\RR}{\mathds{R}}
\newcommand{\Cs}{\mathscr{C}}
\newcommand{\Gs}{\mathscr{G}}
\newcommand{\PP}{\mathds{P}}
\newcommand{\Fs}{\mathscr{F}}
\newcommand{\As}{\mathscr{A}}
\newcommand\FF{\mathcal{F}}
\newcommand\FFF{\mathds{F}}
\newcommand\tN{\tilde{N}}
\newcommand\N{\mathfrak{N}}
\DeclareMathOperator{\sgn}{sgn}
\newcommand\Y{\mathcal{Y}}
\newcommand\X{\mathcal{X}}
\newcommand\T{\mathcal{T}}

\newcommand{\kbar}{\overline k}
\newcommand{\Xbar}{\overline X}
\newcommand{\Rbar}{\overline R}
\newcommand{\Ybar}{\overline Y}
\newcommand{\mt}{t}
\newcommand{\m}{m}
\newcommand{\us}{\underline u}
\newcommand{\coord}{a}
\DeclareMathOperator{\identity}{id}

\begin{document}

\title{The split torsor method for Manin's conjecture}

\author{Ulrich Derenthal} 

\address{Leibniz Universit\"at Hannover, Institut f\"ur Algebra, Zahlentheorie und Diskrete
  Mathematik, Welfengarten 1, 30167 Hannover, Germany}
\email{derenthal@math.uni-hannover.de}

\author{Marta Pieropan} 

\address{EPFL SB MATH CAG, B\^at. MA, Station 8, 1015 Lausanne, Switzerland}
\email{marta.pieropan@epfl.ch}

\date{May 5, 2020}

\subjclass[2010]{11D45 (11G35, 14G05)}

\setcounter{tocdepth}{1}

\begin{abstract}
  We introduce the split torsor method to count rational points of bounded
  height on Fano varieties.  As an application, we prove
  Manin's conjecture for all nonsplit quartic del Pezzo surfaces of type
  $\Athree+\Aone$ over arbitrary number fields. The counting problem on the
  split torsor is solved in the framework of o-minimal structures.
\end{abstract}

\maketitle

\tableofcontents

\section{Introduction}

\subsection{Manin's conjecture}

Manin's conjecture makes a precise prediction for the number of rational
points of bounded anticanonical height on (possibly singular) Fano varieties
over number fields \cite{MR89m:11060,MR1340296,MR1620682,MR2019019}. For
simplicity, we state it for an anticanonically embedded del Pezzo surface
$S \subset \PP^n_K$ of degree $\ge 3$ with at most
$\mathbf{ADE}$-singularities over a number field $K$: Let $U$ be the
complement of the lines on $S$, and let $H$ be the height function on $S(K)$
induced by the exponential Weil height on $\PP^n(K)$.  Then we expect that
\begin{equation*}
  N_{U,H}(B) \coloneqq |\{\xx \in U(K) : H(\xx) \le B\}| = c_{S,H} B(\log B)^{\rho(\tS)-1}(1+o(1))
\end{equation*}
as $B \to \infty$, where $\rho(\tS)$ is the Picard number of
a minimal desingularization $\tS$ of $S$ over $K$, and $c_{S,H}$ is a positive constant that 
can be written as an explicit prescribed 
product of invariants of $K$ and $S$, including local densities on $\tS$.

Two major approaches to Manin's conjecture are harmonic analysis on adelic
points and the universal torsor method.  The \emph{harmonic analysis approach}
can only be applied to equivariant compactifications of algebraic groups,
which include just a few types of del Pezzo surfaces, mostly of high degree;
see \cite{MR2753646,MR3333982} for their classification. For these, it can
give very general results, covering all varieties in a given class over
arbitrary number fields; see \cite{MR89m:11060,MR1620682,MR1906155}, for
example.  The \emph{universal torsor method} can in principle be applied to
all Fano varieties, but in each instance so far, it treats only a single
example or a smaller family of varieties, with new difficulties arising in
every case.  It has been applied mostly to split varieties (that is, varieties
for which the Picard group and the geometric Picard group coincide; for del
Pezzo surfaces of degree $\le 7$, this is equivalent to the lines being
defined over the ground field) over $\QQ$.  The universal torsor method over
arbitrary number fields has been developed in \cite{MR3552013}.  Due to the
possible presence of nonsplit tori, it does not seem to be the right way to
treat nonsplit varieties (that is, varieties with a nontrivial Galois action
on their geometric Picard group), as we explain below.

Here, we propose a new variation, the \emph{split torsor method} for
nonsplit varieties, 
and we prove Manin's conjecture for all the nonsplit forms of the variety 
considered in \cite{MR3552013} over arbitrary number fields.

\subsection{Torsors}

Let $X$ be a projective variety over a number field $K$ with algebraic closure
$\Kbar$.  Colliot-Th\'el\`ene and Sansuc studied torsors $\pi\colon Y \to X$
under (quasi)tori $T$ over $X$ \cite{MR899402}. Such a torsor has a
\emph{type}, namely the homomorphism of $\Gal(\Kbar/K)$-modules
\begin{equation*}
  \widehat T\coloneqq\Hom_{\Gal(\Kbar/K)}(T_\Kbar,\GmKbar) \to \Pic(X_\Kbar)
\end{equation*}
that sends a character $\chi \colon T_\Kbar \to \GmKbar$ to the class of the
extension $\chi_*Y$. If $\Pic(X_\Kbar)$ is free and finitely generated, then
one may consider torsors under the N\'eron--Severi torus $T$ with
$\widehat T = \Pic(X_\Kbar)$ whose type is the identity. These are called
\emph{universal torsors} and are particularly useful, for example, to study
the Hasse principle and weak approximation on geometrically rational
varieties.

\subsection{Universal torsor method}
The basic idea of the universal torsor method for Manin's conjecture  
consists in  constructing a
model $\T$ of the N\'eron--Severi torus $T$ over 
 the ring of integers $\OO_K$ of $K$ and a torsor
$\pi \colon \Y \to \X$ under $\T$ that is an $\OO_K$-model of a universal torsor
$\pi \colon Y \to X$, and using the induced map
\begin{equation}\label{eq:torsor_points}
  \pi \colon \Y(\OO_K) \to \X(\OO_K) \subseteq X(K)
\end{equation}
to parameterize rational points on $X$ by the $\T(\OO_K)$-orbits on
$\Y(\OO_K)$.  After lifting the height function, one can use techniques of
analytic number theory to count these orbits of bounded height.

If $X$ is split and $K=\QQ$, the map \eqref{eq:torsor_points} is surjective and has finite fibers.

If $X$ is nonsplit, the map \eqref{eq:torsor_points} does not need to be
surjective even if $K=\QQ$, see for example \cite[\S4]{MR2874644}.
Nonetheless, torsor parameterizations have been used to verify Manin's
conjecture for some nonsplit varieties over $\QQ$ that split over
$\QQ(\sqrt{-1})$.  See \cite{MR2099200} for a smooth quintic del Pezzo
surface, \cite{MR2838351} for a smooth quartic del Pezzo surface,
\cite{MR2373960} for a quartic del Pezzo surface of type $\Dfour$ and
\cite{MR2874644,MR3011504,MR3103132,MR3517531} for quartic del Pezzo surfaces
of type $2\Aone$ (Ch\^atelet surfaces). For example in
\cite[Corollary~5.25]{MR2874644}, several (but finitely many) twists of a
universal torsor are used to parameterize the rational points on the
underlying variety. For all such varieties, the N\'eron--Severi torus $T$ is
not split, but $\T(\ZZ)$ is still finite.

Finiteness of $\T(\ZZ)$ also holds for varieties over $\QQ$ that split over
imaginary quadratic fields $\QQ(\sqrt{a})$, but here one may have to deal with
bad reduction at primes dividing $a$; apparently no examples with $a<-1$ have
been worked out.

Already in the case of nonsplit varieties over $\QQ$ that split over a real
quadratic field $\QQ(\sqrt{a})$, the set $\T(\ZZ)$ is infinite. A typical
example is $\T \cong \T_a \times \GmZ^{\rho-1}$, where $\T_a$ is a model of
the norm-$1$-torus for $\QQ(\sqrt{a})/\QQ$ with
\begin{equation*}
  \T_a(\ZZ) = \{(z_1,z_2)\in \ZZ^2 : z_1^2-az_2^2=1\}.
\end{equation*}
This holds for example for the surfaces treated in the present paper if one
takes $a>0$ in Theorem \ref{thm:number_fields} below.
  
The main difficulty in the application of the universal torsor method for
nonsplit varieties that split over real quadratic fields would be the
construction of a fundamental domain for the action of $\T(\ZZ)$ on a model
$\Y \to \X$ of a universal torsor $Y \to X$. An additional difficulty could be
the need for twists of these torsors to cover $X(\QQ)$. Apparently this
strategy has never been implemented.
  
\subsection{Split torsor method}
For nonsplit varieties, we propose to work with the split torsor instead of
universal torsors. The \emph{split torsor} $\pi' \colon Y' \to X$ is a torsor
under the split torus $T'$ with character group $\widehat T' = \Pic(X)$ whose
type is the natural map $\Pic(X) \to \Pic(X_\Kbar)$; it is unique up to
isomorphism and the map $Y'(K)\to X(K)$ is surjective.
   
To make the split torsor $\pi' \colon Y' \to X$ explicit, we have developed a
theory and constructions of Cox rings of arbitrary type in
\cite{arXiv:1408.5358}. We can construct an $\OO_K$-model
$\pi'\colon\Y' \to \X$ using the theory that we develop in
Section~\ref{sec:integral_models}.  Assuming $K=\QQ$ and
$\Pic(X) \cong \ZZ^{\rho'}$, the map
\begin{equation}\label{eq:splittorsor}
  \pi' \colon \Y'(\OO_K) \to \X(\OO_K).
\end{equation}
is surjective, and it is a $(2^{\rho'}:1)$-map.
  
In \cite{MR2373960}, a parameterization of the rational points on a quartic
del Pezzo surface of type $\Dfour$ over $\QQ$ that splits over $\QQ(i)$ is
produced by elementary manipulations of the defining equations, without a
geometric interpretation. Using the techniques from \cite{arXiv:1408.5358} and
Section~\ref{sec:integral_models}, one can show that a model of the split
torsor gives the same parameterization; see
\cite[Example~5.1]{arXiv:1408.5358}.

The split torsor method not only avoids the difficulties due to the 
  splitting fields of quasitori, but it produces a parameterization
  that enables the application  in the nonsplit setting of known counting 
  techniques typically used on split varieties.

The difficulty of the resulting counting problem depends on the variety in
question. We expect that the primes of $\OO_K$ that are ramified in the
splitting field of $X$ will typically be places of bad reduction for the model $\X$ and
cause considerable additional complications, as it happens for the 
varieties considered in this paper. In our case, the first new difficulty
appears at the M\"obius inversion step, introducing an extra 
summation over congruence classes. This is eventually encoded in a fairly complicated 
arithmetic function that we study via related Hecke $L$-functions.

If $K$ is an arbitrary number field, the maps \eqref{eq:torsor_points} and
\eqref{eq:splittorsor} do not need to be surjective nor have finite
fibers. For split varieties, both problems have been addressed and solved in
\cite{MR3552013} via twists of torsors and fundamental domains for the action
of $\T(\OO_K)$. We show that a similar approach works also for nonsplit
varieties.

\subsection{Our application}

We investigate Manin's conjecture for quartic del Pezzo surfaces with
singularities of type $\Athree+\Aone$ over an arbitrary number field $K$.  In
Section~\ref{sec:dPsurfaces}, we show that each such surface is isomorphic
over $K$ to precisely one of the surfaces $S_a \subset \PP^4_K$ defined by
\begin{equation}\label{eq:equationsSa}
  S_a \colon x_0x_4+x_1^2-ax_3^2=x_2x_3-x_4^2 = 0,
\end{equation}
where $a$ runs through a set of representatives of $K^\times/K^{\times 2}$.

We assume that  $a$ is an algebraic integer. We could further assume that $a$ is squarefree (that
is, there is no $c \in \OO_K \smallsetminus \OO_K^\times$ with $c^2 \mid a$), but
this assumption would not simplify our work since there still could be prime
ideals $\p$ in $\OO_K$ with $\p^2 \mid a$.

The two singularities on $S_a$ are $(1:0:0:0:0)$ of type $\Athree$ and
$(0:0:1:0:0)$ of type $\Aone$. Let $\sqrt{a} \in \Kbar$ be a square root of
$a$.  The surface $S_a$ contains three lines: $\{x_1=x_3=x_4=0\}$ defined over
$K$, and $\{x_1+\sqrt{a}x_3=x_2=x_4=0\}$ and $\{x_1-\sqrt{a}x_3=x_2=x_4=0\}$
defined over $K(\sqrt{a})$.

If $a$ is a square, the surface $S_a \cong S_1$ is split over $K$. Manin's
conjecture for $S_1$ is proved in \cite[\S 8]{MR2520770} over $\QQ$, in
\cite[\S 2]{MR3181632} over imaginary quadratic fields and in \cite{MR3552013}
over arbitrary number fields.  If $a$ is not a square, the surface $S_a$ is
nonsplit over $K$ and splits over $K(\sqrt a)$; our main result is a proof of
Manin's conjecture for all these surfaces over an arbitrary number field $K$.

The following arithmetic function plays an important role in our analysis and
appears in the nonarchimedean local densities in the constant $c_{S_a,H}$.
For every ideal $\qfr$ of $K$, let
\begin{equation}\label{eq:eta}
  \eta(\qfr;a)\coloneqq\left|\bigwhere{\rho \pmod{\qfr\gfr^{-1}\gfr'}}{
      &\rho\OO_K+\qfr\gfr^{-1}\gfr'=\gfr',\ \congr{\rho^2}{a}{\qfr},\\
      &\gfr=\qfr+a\OO_K,
      \ \gfr'=\prod_\p \p^{\ceil{v_\p(\gfr)/2}}}\right|.
\end{equation}
As we see in Section \ref{sec:arithmetic_functions},
it is closely related to the Legendre symbol \cite[\S V.3]{MR1697859}, which
is defined for prime ideals $\p \nmid 2a$ as
\begin{equation}\label{eq:legendre}
  \legendre{a}{\p}=
  \begin{cases}
    1, & \text{$a$ is a quadratic residue mod $\p$},\\
    -1, & \text{$a$ is a quadratic nonresidue mod $\p$}.
  \end{cases}
\end{equation}
We denote by $\pi_\p$ a uniformizer for the completion $K_\p$ and observe that
the symbol
\begin{equation*}
  \legendre{a/\pi_\p^{v_\p(a)}}{\p}
\end{equation*}
is well defined when $v_\p(a)$ is even and $\p \nmid 2$, as $a/\pi_\p^{v_\p(a)}$ is a unit in $K_\p$.  We refer to
Section~\ref{sec:notation} for the standard notation that we use in the
description of the leading constant $c_{S_a,H}$ in the theorem below.

\begin{theorem}\label{thm:number_fields}
  Let $K$ be a number field and let $H$ be the exponential Weil height
  on $\PP^4_K$. Assume that $a \in \OO_K$ is not a square. Let
  $S_a \subset \PP^4_K$ be the singular quartic del Pezzo surface
  defined by (\ref{eq:equationsSa}).
  Let $U$ be the complement of the lines on $S_a$.
  Let $\epsilon>0$. For $B \to \infty$, we have
  \begin{equation*}
    N_{U,H}(B) = c_{S_a,H} B(\log B)^4 + O_{\epsilon,a,K}(B(\log B)^{4-\frac 1 d+\epsilon}).
  \end{equation*}
  Here,
  \begin{equation*}
    c_{S_a,H} \coloneqq \alpha_{S_a} \left(\frac{2^{r_1}(2\pi)^{r_2}h_KR_K}{|\mu_K|\cdot \sqrt{|\Delta_K|}}\right)^5 \frac{1}{|\Delta_K|}
    \left(\prod_{v \in \Omega_K} \omega_v(S_a)\right),
  \end{equation*}
  where
  \begin{align*}
    \alpha_{S_a}
    &\coloneqq \frac{1}{1728},\\
    \omega_v(S_a)
    &\coloneqq
      \begin{cases}
       \left(1-\frac{1}{\N\p}\right)^5\left(1+\frac{5+r_a(\p)}{\N\p}+\frac{1}{\N\p^2}\right),
      &\text{$v=\p$ finite,} \\     
       \frac{3}{2} \iiint_{\{N_v(y_5,y_6,y_7) \le 1\}\subset\RR^3} \dd y_5\,\dd y_6\,\dd y_7,
        &\text{$v$ real,}\\
        \frac{12}{\pi} \iiint_{\{N_v(y_5,y_6,y_7) \le 1\}\subset\CC^3} \dd y_5\,\dd y_6\,\dd y_7,
        &\text{$v$ complex,}
    \end{cases}
  \end{align*}
  with
  \begin{align*}
    r_a(\p)&\coloneqq
    \begin{cases}
      \legendre{a}{\p}, &\p \nmid 2a,\\
      1-\frac{1}{\N\p^{\floor{v_\p(a)/2}}}\left(1+\frac{1}{\N\p}\right),
      &\text{$v_\p(a)$ odd,}\\
      1-\frac{1}{\N\p^{v_\p(a)/2}}\left(1-\legendre{a/\pi_\p^{v_\p(a)}}{\p}\right), &
      \p \nmid 2,\ \text{$v_\p(a)$ even,}\\
      1-\frac{1}{\N\p^{v_\p(a)/2}} (2-s_a(\p)),
      &\p \mid 2,\ \text{$v_\p(a)$ even,}
    \end{cases}\\
  \end{align*}
  \begin{align*}s_a(\p)&\coloneqq\left(1-\frac{1}{\N\p}\right)\sum_{k=0}^{v_\p(4)-1}
       \frac{\eta(\p^{v_\p(a)+k+1};a)}{\N\p^k}+\frac{\eta(\p^{v_\p(4a)+1};a)}{\N\p^{v_\p(4)}},
  \end{align*}
  and, for $v \mid \infty$ and $(y_5,y_6,y_7) \in K_v^3$,
  \begin{equation*}
    N_v(y_5,y_6,y_7)\coloneqq\max\{|y_6(ay_6^2-y_7^2)|_v,|y_5y_6y_7|_v,|y_5^3|_v,|y_5y_6^2|_v,|y_5^2y_6|_v\}.
  \end{equation*}
  This agrees with the conjectures of Manin and Peyre \cite[Formule empirique 5.1]{MR2019019}.
\end{theorem}

\begin{remark}
  In the case $K=\QQ$, we may assume that $a\ne 1$ is a squarefree integer. Then
  \begin{equation*}
    \omega_p(S_a)=\left(1-\frac 1 p\right)^5\cdot
    \begin{cases}
      1+\frac{5}{p}, & p \mid a,\\
      1+\frac{6}{p}+\frac{1}{p^2}, & p\ne 2,\ \legendre{a}{p}=1,\\
      1+\frac{4}{p}+\frac{1}{p^2}, & p\ne 2,\ \legendre{a}{p}=-1,\\
      \frac{17}{4}, & p=2,\ \congr{a}{1}{8},\\
      \frac{15}{4}, & p=2,\ \congr{a}{5}{8},\\
      \frac{7}{2}, & p=2,\ \congr{a}{3,7}{8}.
    \end{cases}
  \end{equation*}
\end{remark}

Considering the $\p$-adic densities in our leading constant, it is not
surprising that the primes $\p \mid 2a$, and in particular those with even
$v_\p(a)$, cause considerable technical difficulties. These appear at many
places of the proof.

\subsection{Strategy of proof}

The proof implements the split torsor method that we have proposed
above. Let $\tS_a$ be the minimal desingularization of $S_a$. Here, the type of the torsor is
\begin{equation*}
  \ZZ^5 \cong \Pic(\tS_a) \hookrightarrow \Pic((\tS_a)_\Kbar) \cong \ZZ^6.
\end{equation*}

In Section~\ref{sec:surfaces_torsors}, we construct a split torsor
$Y' \to \tS_a$ over $K$ and an $\OO_K$-model $\Y' \to \tSS_a$ using the
techniques developed in Section~\ref{sec:integral_models} and in
\cite{arXiv:1408.5358}.  In Section~\ref{sec:parameterization_concrete}, we
use $\Y' \to \tSS_a$ and its $h_K^5$ twists (where $h_K$ is the class number
of $K$) to produce an explicit parameterization of $U(K)$ in terms of
coordinates $(a_1,\dots,a_8)$ satisfying the torsor equation
\begin{equation}\label{eq:torsor-intro}
  a_1a_8+a_7^2-aa_2^4a_3^2a_4^6a_6^2=0,
\end{equation}
coprimality conditions and height conditions up to 
the action of $\T'(\OO_K)\cong (\OO_K^\times)^5$.

As in \cite{MR3552013}, the five
degrees of freedom coming from this action allow us to choose a
fundamental domain  such that
all conjugates of $a_1,\dots,a_4$ and of the height function have
roughly the same size; see Section~\ref{sec:fundamental_domain}. 
For fixed $a_1,\dots,a_4$,
the volume of the fundamental domain 
cut by the height conditions for the remaining variables  is roughly
$\prod_{v \mid \infty} \omega_v(S_a)\cdot B/|N(a_2a_3a_4)|$, where $N$ denotes the norm of $K/\QQ$.  In
Section~\ref{sec:moebius}, we remove the coprimality conditions on
$a_5,a_6,a_7$ by M\"obius inversions and interpret the torsor equation
(\ref{eq:torsor-intro}) as a congruence modulo $a_1$, eliminating
$a_8$. This is much more subtle than in the split case and isolates a
summation over $\rho$ as in (\ref{eq:eta}).

For fixed $a_1,\dots,a_4,B$, the first summation in
Section~\ref{sec:first_summation} estimates the number of $a_5,a_6,a_7$
satisfying our congruence. It has many steps. In
Section~\ref{sec:summation_lemma}, we prove a general summation lemma that
will be applied to all error terms in the first summation. In
Section~\ref{sec:small_conjugates}, we remove the ``spikes'' in our counting
domain where $a_6$ or $a_7$ have small conjugates.  As in \cite{MR3552013}, we
show that the counting problem for $(a_5,a_6,a_7)$ consists in estimating the
number of lattice points lying in a set definable in Wilkie's o-minimal
structure $\RR_{\exp}$ \cite{MR1398816}.  Therefore, we can apply \cite{MR3264671}; the error
terms depend on the volumes of the projections of this set to coordinate
hyperplanes, which we estimate in Section~\ref{sec:volumes_of_projections}.
In Section~\ref{sec:first_summation_completion}, we estimate the error term
coming from the ``spikes''.  This completes the first summation, which shows
roughly that after turning the summation over $a_1,\dots,a_4$ and the $h_K^5$
twists into a summation over ideals $\afr_1,\dots,\afr_4$, we have
\begin{equation*}
  N_{U,H}(B) \sim \frac{\rho_K}{3|\Delta_K|}\cdot
  \prod_{v \mid \infty} \omega_v(S_a)\cdot B \cdot \sum_{\afr_1,\dots,\afr_4}
  \frac{\vartheta_1(\afr_1,\dots,\afr_4)}{\N(\afr_1\cdots\afr_4)};
\end{equation*}
see Proposition~\ref{prop:first_summation}. Here, $\vartheta_1$ is an
arithmetic function coming from the coprimality conditions and involving the
function $\eta(\cdot ;a)$ defined in (\ref{eq:eta}).

In the split case, the remaining summations in \cite[\S 12]{MR3552013} are a
straightforward application of the theory developed in \cite[\S 7]{MR3269462}.
In our nonsplit case, we must study averages of arithmetic functions involving
the Legendre symbol~(\ref{eq:legendre}); this is done in
Section~\ref{sec:averages} using the corresponding Hecke $L$-function, which
we discuss in Section~\ref{sec:arithmetic_functions}, combined with Perron's
formula. In Section~\ref{sec:remaining_summations_completion}, we estimate the
sum over $\afr_1,\dots,\afr_4$ as
\begin{equation*}
  3\rho_K^4 \cdot \alpha_{S_a} \cdot (\log B)^4 \cdot \prod_\p
  \omega_\p(S_a),
\end{equation*}
which completes the proof of the asymptotic formula.

In Section~\ref{sec:constant}, we show that our leading constant agrees with
Peyre's prediction in \cite{MR2019019}. We study the convergence factors in
Section~\ref{sec:convergence_factors} and compute the local densities in
Section~\ref{sec:local_densities}. Again, the most difficult places are
$\p \mid 2a$ with even $v_\p(a)$.

\subsection{Notation}\label{sec:notation}

Additionally to the notation introduced in the course of this introduction, we
will use the following: For the number field $K$, the residue of the Dedekind
zeta function at $s=1$ is
\begin{equation}\label{eq:rho_K_def}
  \rho_K \coloneqq \frac{2^{r_1}(2\pi)^{r_2} R_K h_K}{|\mu_K|\cdot |\Delta_K|^{1/2}},
\end{equation}
where $r_1$ is its number of real embeddings, $r_2$ its number of complex
embeddings, $R_K$ its regulator, $h_K$ its class number, $\mu_K$ its group of
roots of unity, and $\Delta_K$ its discriminant. Let $d = [K:\QQ]$ be its degree.

Let $\Omega_K$ be its set of places. For $v \in \Omega_K$, let $K_v$ be the
corresponding completion of $K$. Let $\Omega_\infty$ be the set of archimedean
places; we also write $v \mid \infty$ for $v \in \Omega_\infty$.  Let
$\Omega_f$ be the set of nonarchimedean places.  Let $U_K$ be a free part of
the unit group (that is, $\OO_K^\times = U_K \times \mu_K$).
  
Let $\Is_K$ be the monoid of nonzero ideals in $\OO_K$ and let
$\Cs \subset \Is_K$ be a system of representatives of the class group of
$\OO_K$. Let $\N\qfr$ be the absolute norm of a (fractional) ideal $\qfr$, and
let $N(a)=N_{K/\QQ}(a)$ be the norm of $a \in K$ over $\QQ$. The symbol $\p$
always denotes a prime ideal of $\OO_K$.  Let
$\mu_K \colon \Is_K \to \{-1,0,1\}$ be the M\"obius function. For
$\qfr \in \Is_K$, let $\omega_K(\qfr)$ be the number of distinct prime ideals
dividing $\qfr$.

For $v \in \Omega_K$ lying above $w \in \Omega_\QQ$, let
$d_v\coloneqq [K_v : \QQ_w]$ be the local degree, and let
$|\cdot|_v = |N_{K_v/\QQ_w}(\cdot)|_w$ be the $v$-adic absolute value, where
$|\cdot|_w$ is the usual $p$-adic or real absolute value on $\QQ_w$. Also for
$v \in \Omega_K$, let $\sigma_v \colon K \to K_v$ be the $v$-adic embedding,
we use the same symbol for the component-wise map
$\sigma_v \colon K^n \to K_v^n$, and denote by
$\sigma \colon K \to \prod_{v \in \Omega_K} K_v$ the diagonal map. For
$a \in K$, we often write $|a|_v$ for $|\sigma_v(a)|_v$.  The exponential Weil
height of a rational point $\xx = (x_0 : \dots : x_n) \in \PP^n(K)$ is
\begin{equation*}
  H(\xx) \coloneqq \prod_{v \in \Omega_K} \max\{|x_0|_v, \dots, |x_n|_v\}.
\end{equation*}

For each $a\in K$, we fix a square root of $a$ in $\Kbar$ and we denote it by $\sqrt a$. 
  
We call $x \in K$ \emph{defined} (or \emph{invertible}) modulo an ideal $\qfr$
if $v_\p(x) \ge 0$ (or $v_\p(x)=0$) for all $\p \mid \qfr$. If $x,y \in K$ are
defined modulo $\qfr$, the notation $\congrfr{x}{y}{\qfr}$ means that
$v_\p(x-y)\ge v_\p(\qfr)$ for all $\p \mid \qfr$.

When we use Landau's $O$-notation and Vinogradov's $\ll$-notation, we mean
that the corresponding inequalities hold for \emph{all} values in the relevant
range. The implied constants may always depend on $K$ and $a$; additional
dependencies are denoted by subscripts.
  
All volumes of subsets of $\RR^n$ are computed with respect to the usual
Lebesgue measure, unless stated otherwise.

\subsection*{Acknowledgements}

The first author was partly supported by grant DE 1646/4-2 of the Deutsche
Forschungsgemeinschaft. Some of this work was done while he was on sabbatical
leave at the University of Oxford. The second author was partly supported by
grant ES 60/10-1 of the Deutsche Forschungsgemeinschaft. The authors are
grateful to D.~R.~Heath-Brown, E.~Sofos and the anonymous referee for their helpful remarks.

\section{Integral models of split torsors}\label{sec:integral_models}

We extend the construction of integral models of split universal 
torsors developed in \cite[\S3]{MR3552013} to (not necessarily universal) torsors under
split tori; see also \cite[\S3]{thesis_pieropan}.  
We refer to \cite{MR899402} for the theory of torsors
and to \cite{arXiv:1408.5358} for the theory of Cox rings.

\begin{construction}\label{hypot3}
  Let $A$ be a noetherian integral domain.  We fix a separable closure $\kbar$
  of the fraction field of $A$.  Let $\Xbar$ be an integral projective
  $\kbar$-variety with a finitely generated Cox ring $\Rbar$ of type
  $\lambda\colon\ZZ^r\to\Pic(\Xbar)$ such that $\lambda(\ZZ^r)$ contains an ample
  divisor class.  Let $\eta_1,\dots,\eta_N$ be $\ZZ^r$-homogeneous elements of
  $\Rbar$ such that
  \begin{equation*}
    \Rbar=\kbar[\eta_1,\dots,\eta_N]/I
  \end{equation*}
  for a homogeneous ideal $I$ as in \cite[Proposition 4.4]{arXiv:1408.5358}.
  Assume that $I$ is generated by polynomials
  $g_1,\dots,g_s\in A[\eta_1,\dots,\eta_N]$.  Let
  \begin{equation*}
    f_1,\dots,f_\mt\in \kbar[\eta_1,\dots,\eta_N]\smallsetminus \sqrt{I}
  \end{equation*}
  be monic monomials such that the complement $\Ybar$ of the closed subset of
  $\Spec\Rbar$ defined by $f_1,\dots, f_{\mt}$ is a torsor of $\Xbar$ of type
  $\lambda$ (cf.~\cite[Corollary 4.3]{arXiv:1408.5358}).

  Let
  \begin{equation*}
    R\coloneqq A[\eta_1,\dots,\eta_N]/(g_1,\dots,g_s)
  \end{equation*}   
  with the $\ZZ^r$-grading induced by the $\ZZ^r$-grading on $\Rbar$.  Assume
  that $(R;f_1,\dots,f_\mt)$ satisfies the following condition:
  \begin{equation}
    \parbox{0.8\linewidth}{for every $i,j\in\{1,\dots,\mt\}$, there is a homogeneous
      invertible element of $R[f_i^{-1}]$ of degree a 
      multiple of $\deg f_j$.}\label{conditionstar}
  \end{equation}
  Let $Y$ be the complement of the closed subset of $\Spec R$ defined by
  $f_1,\dots,f_\mt$.  For $i\in\{1,\dots,\mt\}$, let
  $U_i\coloneqq \Spec R[f_i^{-1}]$, let $R_i$ be the degree-$0$-part of the
  ring $R[f_i^{-1}]$ and let $V_i\coloneqq \Spec(R_i)$.  Let $X$ be the
  $A$-scheme obtained by gluing $\{V_i\}_{1\leq i\leq \mt}$, and let
  $\pi \colon Y\to X$ be the morphism induced by the inclusions
  $R_i\to R[f_i^{-1}]$ for $i\in\{1,\dots,\mt\}$.
\end{construction}

\begin{remark}\label{rem:model_construction}
  As in \cite[Construction 3.1]{MR3552013}, $Y_{\kbar}\cong \overline Y$,
  $X_{\kbar}\cong\overline X$ and $\pi$ is an $A$-model of the torsor morphism
  $\overline Y\to\overline X$.  The $\ZZ^r$-grading on $R$ induces an action
  of $\GG_{m,A}^r$ on $Y$.  The morphism $\pi$ is surjective and of finite
  presentation, and it gives an $X$-torsor under $\GG_{m,A}^r$ (compatible with
  the structure of $\overline X$-torsor of type $\lambda$ on $\overline Y$) if
  and only if $\pi$ is flat and the morphism of schemes
  $\psi\colon\GG^r_{m,A}\times_{\Spec A} Y \to Y \times_{X}Y$ defined by
  $(\us,\underline{\coord})\mapsto(\us*\underline\coord,\underline\coord)$ is
  an isomorphism.  The same proof as in \cite[Theorem 3.3]{MR3552013} shows
  that $\pi$ is an $X$-torsor under $\GG_{m,A}^r$ if $(R; f_1,\dots,f_\mt)$
  satisfies the following condition:
  \begin{equation}
    \parbox{0.8\linewidth}{the degrees of the homogeneous invertible elements of
      $R[f_i^{-1}]$ generate $\ZZ^r$ for all
      $i\in\{1,\dots,\mt\}$.}\label{picardgeneration}
  \end{equation}
\end{remark}

\begin{remark}\label{udesck}
  If $A$ is a field, then $\pi\colon Y\to X$ is a torsor of type
  $\lambda \colon \ZZ^r\to\Pic(\Xbar)$ (with trivial $\Gal(\kbar/A)$-action on
  $\ZZ^r$) regardless of condition \eqref{picardgeneration} by \emph{fpqc}
  descent.
\end{remark}

\begin{remark}
  Let $k$ be a finite extension of the fraction field of $A$ contained in
  $\overline k$. If $\widetilde X$ is a $k$-variety such that
  ${\widetilde X}_{\overline k}\cong\overline X$ and $\lambda$ factors through
  $\Pic(\widetilde X)\to\Pic(\overline X)$, the Cox ring $\Rbar$ has a natural
  structure of $\Gal(\kbar/k)$-equivariant Cox ring of
  ${\widetilde X}_{\overline k}$ in the sense of \cite[Definition
  3.1]{arXiv:1408.5358}, and we can choose
  $\eta_1,\dots,\eta_N\in \Rbar^{\Gal(\kbar/k)}$. Then $R_{k}$ is a Cox ring
  of $\widetilde X$ of type
  $\lambda\colon \ZZ^r\to\Pic({\widetilde X}_{\overline k})$ and
  $Y_k\to \widetilde X$ is a torsor of $\widetilde X$ of type
  $\lambda\colon \ZZ^r\to\Pic({\widetilde X}_{\overline k})$. In particular,
  $X_k\cong \widetilde X$.
\end{remark}

\begin{remark}\label{rem:same_statements}
  Straightforward generalizations of \cite[Lemma 3.5, Proposition
  3.6]{MR3552013} hold in the setting of Construction \ref{hypot3} with the
  very same statement and the very same proof.
\end{remark}

For the sake of precision we report the results analogous to
\cite[Propositions 3.7, 3.8]{MR3552013}.

\begin{prop}\label{prop:model_properties}
  Assume the setting of Construction \ref{hypot3}.
  \begin{enumerate}[label=(\roman*), ref=(\roman*)]
  \item\label{it:model_projective}
    Let $C_{\kbar}$ and $C_A$ be the ideals of 
    $\kbar[\eta_1,\dots,\eta_N]$ and $A[\eta_1,\dots,\eta_N]$, 
    respectively, generated by $f_1,\dots,f_\mt,g_1,\dots,g_s$.  
    If $f_1,\dots,f_\mt$ have all the same degree $\m\in\ZZ^r$
    such that $\lambda(\m)$ is very ample, and
    \begin{equation*}
      \sqrt{C_{\kbar}}\cap A[\eta_1,\dots,\eta_N]=\sqrt{C_A},
    \end{equation*}
    then $X$ is projective over $A$.
  \item\label{it:model_smooth} Assume that $A$ is a Dedekind domain,  $R$ is an
    integral domain, $\Spec(R)\to\Spec(A)$ has geometrically integral fibers,
    and $\pi$ is flat (the last holds, for example, if
    \eqref{picardgeneration} is satisfied). If the Jacobian matrix
    \begin{equation*}
      \left(\frac{\partial g_i}{\partial\eta_j}
        (\underline{\coord})\right)_{\substack{1\leq i\leq s\\ 1\leq j\leq N}}
    \end{equation*} 
    has rank $N-\dim \overline X-r$ for all
    $\underline{\coord}\in Y(\overline{k(\p)})$ and $\p\in\Spec(A)$, where
    $\overline{k(\p)}$ is an algebraic closure of the residue field $k(\p)$,
    then $X$ is smooth over $A$.
  \end{enumerate}
\end{prop}

\begin{proof}
  For projectivity it suffices to replace \cite[Corollary 1.6.3.6]{MR3307753}
  by \cite[Corollary 4.3]{arXiv:1408.5358} and \cite[Lemma 3.5]{MR3552013} by
  its straightforward analog (cf. Remark \ref{rem:same_statements}) in the
  proof of \cite[Proposition 3.8]{MR3552013}.  Smoothness is proved by the
  same argument used in \cite[Proposition 3.7]{MR3552013}.
\end{proof}

\begin{remark}\label{integrality_remark}
  Assume that $\overline k$ has characteristic 0, that $\Xbar$ is
  normal and that $\Pic(\Xbar)$ is finitely generated.  By
  \cite[Theorem 1.5.1.1]{MR3307753}, a Cox ring of $\Xbar$ of identity
  type is an integral domain, hence also $\Rbar$ is an integral
  domain, as the grading group of $\Rbar$ is free.  Then $R$ is an
  integral domain if and only if
  $I\cap A[\eta_1,\dots,\eta_N]=(g_1,\dots,g_s)$.
\end{remark}

\section{Arithmetic functions}\label{sec:arithmetic_functions}

In this section, we study the function $\eta(\cdot;a)$ introduced in
\eqref{eq:eta}, and the Hecke $L$-function of the Legendre symbol, which plays
an important role in Section \ref{sec:averages} and in the study of the
Tamagawa number in Section \ref{sec:convergence_factors}.

Let $K$ be a number field and fix an $a \in \OO_K$ that is not a square.
Recall the Legendre symbol from \eqref{eq:legendre}.

\begin{lemma}\label{lem:eta_properties}
  The function $\eta(\qfr;a)$ defined in (\ref{eq:eta}) is multiplicative in
  $\qfr$. For a prime ideal $\p$ and $k \in \ZZp$, we have 
  \begin{equation*}
    \eta(\p^k;a)\!=\!
    \begin{cases}
      1+\legendre{a}{\p}, & \p \nmid 2a,\\
      1, & k \le v_\p(a),\\
      0, & k > v_\p(a),\ \text{$v_\p(a)$ odd},\\
      1+\legendre{a/\pi_\p^{v_\p(a)}}{\p}, &k > v_\p(a),\ \text{$v_\p(a)>0$ even},\
      \p \nmid 2,\\
      \le \N\p^{k-v_\p(a)}-\N\p^{k-v_\p(a)-1}, & v_\p(a)\!<\! k \!\le\!
      v_\p(4a)+1,\,\text{$v_\p(a)$ even},\,\p \mid 2,\\
      \eta(\p^{v_\p(4a)+1};a), & k > v_\p(4a)+1,\ \text{$v_\p(a)$ even},\ \p
      \mid 2.
    \end{cases}
  \end{equation*}
  Furthermore,
  \begin{equation*}
    \eta(\qfr;a) \ll 2^{\omega_K(\qfr)}.
  \end{equation*}
\end{lemma}

\begin{proof}
  The multiplicativity of $\eta(\qfr;a)$ follows from the Chinese remainder
  theorem. For $\p \nmid 2a$, we have $\eta(\p^k;a)=1+\legendre{a}{\p}$
  for $k=1$ by the definition of $\eta$, and for $k > 1$ by Hensel's lemma.
  For the rest of the proof, we assume $\p \mid 2a$ and write
  $n=v_\p(a)$.

  For $n>0$ and $k\le n$, we have
  \begin{equation*}
    \eta(\p^k;a) = |\{\rho \pmod {\p^{\ceil{k/2}}} :
    \rho\OO_K+\p^{\ceil{k/2}}=\p^{\ceil{k/2}},
    \ \congr{\rho^2}{a}{\p^k}\}|=1.
  \end{equation*}
  
  For odd $n \ge 1$ and $k>n$, we have
  \begin{equation*}
    \eta(\p^k;a) = |\{\rho \pmod {\p^{k-\floor{n/2}}} :
    \rho\OO_K+\p^{k-\floor{n/2}}=\p^{\ceil{n/2}},
    \ \congr{\rho^2}{a}{\p^k}\}|=0   
  \end{equation*}
  because the first condition implies $v_\p(\rho)=\ceil{n/2}$, hence
  $v_\p(\rho^2)=n+1$, which contradicts the second condition.

  For even $n \ge 0$ and $k>n$, we have
  \begin{equation*}
    \eta(\p^k;a) = |\{\rho \pmod {\p^{k-n/2}} :
    v_\p(\rho) = n/2,\ \congr{\rho^2}{a}{\p^k}\}|.
  \end{equation*}
  For $\p \nmid 2$, we observe that for any uniformizer $\pi_\p$ in $\OO_{\p}$,
  \begin{equation*}
    \eta(\p^{n+1};a)=1+\legendre{a/\pi_\p^{v_\p(a)}}{\p}.
  \end{equation*}
  For $\p \mid 2$, we have $\eta(\p^{n+1};a) \le \N\p-1$ since there are only
  $\N\p-1$ possibilites for $\rho \pmod{\p^{n/2+1}}$ satisfying
  $v_\p(\rho)=n/2$. Every such $\rho$ has at most $\N\p^{k-n-1}$ lifts
  $\rho' \pmod{\p^{k-n/2}}$, hence
  $\eta(\p^k;a) \le \N\p^{k-n-1}\eta(\p^{n+1};a)$ for $k>n+1$.

  By Hensel's lemma, each $\rho_0 \pmod{\p^{k_0-n/2}}$
  satisfying the congruence $\congr{\rho^2}{a}{\p^{k_0}}$ lifts to a unique
  solution $\rho \pmod{\p^{k-n/2}}$ satisfying $\congr{\rho^2}{a}{\p^k}$
  for every $k>k_0$ as soon as $k_0 = 2(v_\p(2\rho)) + 1 = v_\p(4a) + 1$.
  
  The upper bound for $\eta(\qfr;a)$ follows from the
  multiplicativity of $\eta$ and the fact that, for $k > 0$, we have
  $\eta(\p^k;a) \le 2$ for all $\p \nmid 2$, while
  $\eta(\p^k;a) \le \N\p^{v_\p(4)+1}$ for $\p \mid 2$.
\end{proof}

\begin{example}
  For $K=\QQ$ and squarefree $a \in \ZZ$, we have
  \begin{equation*}
    \eta(p^k;a) =
    \begin{cases}
      1+\legendre{a}{p}, & p \nmid 2a,\\
      1, & p \mid a,\ k=1,\\
      1, & p = 2,\ k=1,\ \congr{a}{1}{2},\\
      2, & p = 2,\ k=2,\ \congr{a}{1}{4},\\
      4, & p = 2,\ k\ge 3,\ \congr{a}{1}{8},\\
      0, & \text{else.}
    \end{cases}
  \end{equation*}
\end{example}

For an ideal $\qfr= \prod_{i=1}^r \p_i^{e_i}$ (with prime ideals $\p_i$ and
$e_i \in \ZZp$) that is coprime to $2a\OO_K$, we have as in
\cite[Definition~VI.8.2]{MR1697859}
\begin{equation*}
  \legendre{a}{\qfr} =  \prod_{i=1}^r \legendre{a}{\p_i}^{e_i}.
\end{equation*}
For $\qfr \in \Is_K$, let
\begin{equation}\label{eq:def_chi}
  \chi(\qfr)=
  \begin{cases}
    \legendre{a}{\qfr}, & \qfr + 2a\OO_K = \OO_K,\\
    0, & \text{else.}
  \end{cases}
\end{equation}

Recall the definition \cite[Definition~X]{MR1544310} of characters modulo an
ideal $\ffr$. In the terminology of \cite[\S 1]{MR0218327},
\cite[VII.6.8]{MR1697859}, these are are \emph{(generalized) Dirichlet
  characters} on the narrow ray class group modulo $\ffr$.  The following
result must be well-known; over $\QQ$, see \cite[\S VI.1.3,
Proposition~5]{MR0344216}.

\begin{lemma}
  The function $\chi$ as in (\ref{eq:def_chi}) is a nonprincipal character
  modulo $8a\OO_K$.
\end{lemma}

\begin{proof}
  Let $\ffr\coloneqq8a\OO_K$. By definition, $\chi$ is a multiplicative function on
  the monoid of all ideals in $\Is_K$ that are coprime to $\ffr$.

  Let $x \in \OO_K$ be totally positive with $\congr{x}{1}{\ffr}$. By reciprocity
  \cite[Theorem~VI.8.3]{MR1697859},
  \begin{equation*}
    \chi(x\OO_K) = \legendre{a}{x} = \legendre{x}{a}
    \prod_{\p \mid 2\infty} \hilbert{x}{a}{\p},
  \end{equation*}
  where $\hilbert{\cdot}{\cdot}{\p}$ is the Hilbert symbol \cite[\S
  V.3]{MR1697859}. Since $x$ is totally positive, $\hilbert{x}{a}{\p}=1$ for
  all $\p \mid \infty$. By Hensel's lemma, $\congr{x}{1}{8}$ implies that $x$
  is a square in $K_\p$ for every $\p \mid 2$, hence $\hilbert{x}{a}{\p}=1$
  for all $\p \mid 2$. Writing $a\OO_K=\prod_{\p\mid a} \p^{e_\p}$, we have
  \begin{equation*}
    \legendre{x}{a} = \prod_{\p\mid a} \legendre{x}{\p}^{e_\p} = 1
  \end{equation*}
  as $\congr{\legendre{x}{\p}}{x^{(\N\p-1)/2}}{\p}$ and $\congr{x}{1}{\p}$ for
  all $\p \mid a$.  In total, this proves $\chi(x\OO_K)=1$.

  Next, we claim that $\chi(\afr)=\chi(\bfr)$ for all ideals
  $\afr \sim \bfr \pmod{\ffr}$ as in \cite[Definition~VIII]{MR1544310} (that is,
  $\afr+\ffr=\bfr+\ffr = \OO_K$ and there are totally positive
  $\alpha,\beta \in \OO_K$ with $\alpha\afr = \beta\bfr$ and
  $\alpha \equiv \beta \equiv 1 \pmod \ffr$). By multiplicativity
  \begin{equation*}
    \chi(\alpha\OO_K)\chi(\afr) = \chi(\alpha\afr) =
    \chi(\beta\bfr) = \chi(\beta\OO_K)\chi(\bfr).
  \end{equation*}
  Since $\chi(\alpha\OO_K)=\chi(\beta\OO_K)=1$ as above, our claim follows.
  
  By the Grunwald--Wang Theorem \cite[\S X]{MR2467155}, 
  the fact
  that $a$ is not a square in $K$ implies that it is not a square in $K_\p$
  for infinitely many $\p$, in particular for some $\p \nmid 2a$, and for
  these $\chi(\p)=\legendre{a}{\p}=-1$. Therefore, $\chi$ is not the principal
  character modulo $\ffr$.
\end{proof}

\begin{lemma}\label{lem:upper_bound_eta_omega_new}
  Assume that $a \in \OO_K$ is not a square. Let $C \ge 0$ and $\epsilon > 0$. For
  $t \ge 0$, we have
  \begin{equation*}
    \sums{\qfr \in \Is_K\\\N\qfr \le t} \eta(\qfr;a)\cdot
    (1+C)^{\omega_K(\qfr)}
    \ll_{C,a,\epsilon} t(\log(t+2))^{C+\epsilon}.
  \end{equation*}
\end{lemma}

\begin{proof}
  Let $\Sigma$ be the sum that we want to estimate. We may assume $t \ge 1$. As in
  \cite[Lemma~2.9]{MR3269462}, we deduce
  \begin{equation*}
    \Sigma \ll t\cdot\exp\left(\sum_{\N\p \le t} \frac{\eta(\p;a)\cdot (1+C)-1}{\N\p}\right).
  \end{equation*}
  Since $\eta(\p;a) = 1+\legendre{a}{\p}$ for $\p \nmid 2a$,
  an application of \cite[Satz~LXXXIV]{MR1544310} to the principal character and to
  $\chi$ as in (\ref{eq:def_chi}) shows that the sum over $\p$ is
  \begin{equation*}
    \ll C\cdot \log \log t + (1+C)\cdot o_a(\log \log t) + O_{C,a}(1),
  \end{equation*}
  where the $O_{C,a}(1)$-term accounts for the summands with $\p \mid
  2a$. Estimating $o_a(\log \log t)$ by
  $\frac{\epsilon}{1+C} \log \log t$ for $t \gg_{C,a,\epsilon} 1$ completes
  the proof.
\end{proof}

\begin{lemma}\label{lem:L-function}
  Let $\chi$ be as in~(\ref{eq:def_chi}).
  \begin{enumerate}[label=(\roman*), ref=(\roman*)]
  \item\label{it:absolute_convergence} For $s \in \CC$ with $\Re(s) > 1$, we have
    \begin{equation*}
      L(s,\chi)\coloneqq\sum_{\qfr \in \Is_K} \frac{\chi(\qfr)}{\N\qfr^s} = \prod_\p \frac{1}{1-\frac{\chi(\p)}{\N\p^s}},
    \end{equation*}
    where the series and the Euler product converge absolutely. It can be
    extended to an entire function on $\CC$. 
  \item\label{it:conditional_convergence} The series $L(s,\chi)$ converges
    conditionally for $\Re(s)>1-\frac{2}{d+1}$. 
  \item\label{it:value_at_1}  We have
    \begin{equation*}
      L(1,\chi) = \prod_\p \frac{1}{1-\frac{\chi(\p)}{\N\p}} =
      \prod_{\p \nmid 2a} \frac{1}{1-\frac{\legendre{a}{\p}}{\N\p}} \ne 0,
    \end{equation*}
    where the product converges conditionally. 
  \item\label{it:vertical_growth} For $\epsilon>0$ and
    $|\Im(s)| \ge 2$ and $\Re(s) \ge -\epsilon$, we have
    \begin{equation*}
      L(s,\chi) \ll_\epsilon |\Im(s)|^{(\frac 1 2+\epsilon)d}.
    \end{equation*}
  \end{enumerate}
\end{lemma}

\begin{proof}
  For \ref{it:absolute_convergence}, see \cite[Definition~XVII,
  Satz~LVIII]{MR1544310}. Since $\chi$ is nonprincipal, the second statement is
  \cite[Satz~LXIII]{MR1544310}.

  For \ref{it:conditional_convergence}, let $\delta\coloneqq\frac{2}{d+1}$. By
  \cite[Satz~XCV]{MR1544310},
  \begin{equation*}
    A(x)\coloneqq\sum_{\N\qfr \le x} \chi(\qfr) = O(x^{1-\delta}).
  \end{equation*}
  Let $\Re(s)>1-\delta$. By partial summation, for $x > 1$,
  \begin{equation*}
    \sum_{\N\qfr \le x} \frac{\chi(\qfr)}{\N\qfr^s} = \frac{A(x)}{x^s} -
    \int_1^x \frac{-s\cdot A(t)}{t^{s+1}} \dd t.
  \end{equation*}
  Therefore,
  \begin{equation*}
    \sum_{\N\qfr \le x} \frac{\chi(\qfr)}{\N\qfr^s}-\int_1^\infty \frac{s\cdot A(t)}{t^{s+1}} \dd t
    \ll \frac{A(x)}{x^s} + \int_x^\infty \frac{s\cdot A(t)}{t^{s+1}} \dd t \ll_s x^{1-\delta-\Re(s)}.
  \end{equation*}
  As $x \to \infty$, this shows that $L(s,\chi)$ converges.

  For \ref{it:value_at_1},
  \cite[Theorem~4]{MR1700882} and \cite[Satz~LXXXIII]{MR1544310} imply that
  \begin{equation*}
    \prod_{\N\p \le x} \frac{1}{1-\frac{\chi(\p)}{\N\p}} =
    L(1,\chi)+O_K((\log x)^{-1}),
  \end{equation*}
  where $L(1,\chi) \ne 0$ by \cite[Satz~LXIV]{MR1544310}.

  For \ref{it:vertical_growth}, see \cite[Satz~LXX]{MR1544310}.
\end{proof}

\section{The surfaces and their torsors}\label{sec:surfaces_torsors}

In this section, we study the nonsplit quartic del Pezzo surfaces with
singularities $\Athree+\Aone$, their Cox rings and their split torsors.

\subsection{The surfaces}\label{sec:dPsurfaces}

\begin{prop}\label{prop:classification}
  Every quartic del Pezzo surface with singularities $\Athree+\Aone$ over $K$
  is isomorphic to precisely one of the surfaces $S_a \subset \PP^4_{K}$
  defined by (\ref{eq:equationsSa}), where $a$ runs through a system of
  representatives of $K^\times/K^{\times 2}$.
\end{prop}

\begin{proof}
  Let $X$ be such a singular surface with minimal desingularization $\tX$. By
  \cite[Proposition 6.1]{MR940430}, $\tX_\Kbar$ contains seven negative
  curves: four $(-2)$-curves $D_1,\dots,D_4$ and three $(-1)$-curves
  $D_5,D_6,D_7$ whose intersections are encoded in the Dynkin diagram:
  \begin{equation}\label{eq:Dynkindiagram}
    \xymatrix@R=0.1in @C=0.1in{
      D_7\ar@{-}[rd]& & & & &\\
      &  \ex{D_2} \ar@{-}[r]& \ex{D_3}\ar@{-}[r] & \ex{D_4}\ar@{-}[r] & D_5\ar@{-}[r]& \ex{D_1},\\
      D_6\ar@{-}[ru]& & & & &
    }
  \end{equation}
  We use the same name for a curve and its strict transforms under
  blow-ups. As discussed in the proof of \cite[Lemma~7.4]{MR940430}, the
  simultaneous contraction of $D_6,D_7$ and the interated contraction of
  $D_5,D_4,D_3$ defines a morphism $\rho \colon \tX \to \PP^2_K$ over $K$. We
  observe that $\rho(D_1)$, $\rho(D_2)$ are lines in $\PP^2_K$ defined over
  $K$, that $\rho(D_3)$, $\rho(D_4)$, $\rho(D_5)$ are the intersection point
  $\rho(D_1)\cap\rho(D_2)$, and that $\rho(D_6)$, $\rho(D_7)$ are distinct points on
  $\rho(D_2)$ over $\Kbar$ that are either conjugate under $\Gal(\Kbar/K)$
  (and therefore defined over a quadratic extension of $K$) or each
  defined over $K$.  We note that $\tX$ can be recovered from the line
  $\rho(D_1)$ and the (possibly conjugate) points
  $\rho(D_6),\rho(D_7)$ by the simultaneous blow-up of $\rho(D_6)$,
  $\rho(D_7)$ and three interated blow-ups that are uniquely determined by
  $\rho(D_1)$, $\rho(D_2)$.

  Let $a\in K$ be such that $\rho(D_6)$, $\rho(D_7)$ are defined over
  $K(\sqrt a)$.  Up to a linear change of coordinates $(z_0:z_1:z_2)$ on
  $\PP^2_K$ over $K$, we may assume that $\rho(D_1)=\{z_0=0\}$,
  $\rho(D_6)=(1:0:\sqrt{a})$, and $\rho(D_7)=(1:0:-\sqrt{a})$. This shows that
  $\tX$ is uniquely determined by the class of $a$ in $K^\times/K^{\times
    2}$. Applying the anticanonical map, the same holds for $X$.

  On the other hand, $S_a$ defined by~(\ref{eq:equationsSa}) is a quartic del
  Pezzo surface of type $\Athree+\Aone$ that splits over $K(\sqrt{a})$ because
  it is isomorphic via a linear change of coordinates over $K(\sqrt{a})$ to
  the split quartic del Pezzo surface of type $\Athree+\Aone$ described in
  \cite[\S3.4]{math.AG/0604194}. Therefore, $X \cong S_a$.
\end{proof}

From now on, we fix an element $a \in \OO_K$ that is not a square, and we
denote $S_a$ by $S$.  Let $\phi\colon\tS\to S$ be a minimal
desingularization. The projection
\begin{equation*}
  S\dashrightarrow \PP^2_K, \quad
  (x_0:x_1:x_2:x_3:x_4)\mapsto (x_3:x_4:x_1)
\end{equation*}
from the line $\{x_1=x_3=x_4=0\}$ maps the other two lines
$\{x_1\pm\sqrt{a}x_3=x_2=x_4=0\}$ to $(1:0:\mp\sqrt{a})$ and induces a
morphism $\rho\colon \tS \to\PP^2_K$ that is precisely the sequence of
blow-ups in the proof of Proposition~\ref{prop:classification}.

\subsection{Cox rings}

By \cite[\S3.4]{math.AG/0604194}, there is a basis
$\ell_0,\dots,\ell_5$ of $\Pic(\tS_{\Kbar})$ such that the
$(-2)$-curves in the Dynkin diagram \eqref{eq:Dynkindiagram} have
classes
\begin{gather*}
[D_1]=\ell_0-\ell_1-\ell_2-\ell_3,\ 
[D_2]= \ell_0-\ell_1-\ell_4-\ell_5, \ 
[D_3]=\ell_1-\ell_2, \ 
[D_4]=\ell_2-\ell_3, \ 
\end{gather*}
and the $(-1)$-curves have classes
\begin{gather*}
[D_5]=\ell_3,\ 
[D_6]=\ell_4,\ 
[D_7]=\ell_5, 
\end{gather*}
in $\Pic(\tS_{\Kbar})$. 
Let $D_8$ and $D_9$ be the strict transforms of the lines $\{z_2+\sqrt az_0=0\}$ and $\{z_2-\sqrt az_0=0\}$ under $\rho$, respectively.
Then
\begin{gather*}
  [D_8]=\ell_0-\ell_5,\ 
  [D_9]=\ell_0-\ell_4, 
\end{gather*}

Let $\Lambda\coloneqq\bigoplus_{i=1}^9\ZZ D_i$. Since the divisors classes $[D_2],\dots,[D_7]$ form a basis of $\Pic(\tS_{\Kbar})$, the kernel $\Lambda_0$ of the group homomorphism $\Lambda\to\Pic(\tS_{\Kbar})$ that sends each $D_i$ to its divisor class is a free group of rank 3 with basis
\begin{align*}
D'_1 &= D_1+D_4+2D_5-(D_2+D_6+D_7),\\
D'_8 &= D_8 - (D_2+D_3+D_4+D_5+D_6),\\
D'_9 &= D_9 - (D_2+D_3+D_4+D_5+D_7).
\end{align*}
Let $\Lambda_0\to\Kbar(\tS_{\Kbar})^\times$ be the group homomorphism defined by 
\begin{equation}\label{eq:coxringcharacter}
D'_1\mapsto\rho^*\left(\frac{z_0}{z_1}\right), \quad
D'_8\mapsto\rho^*\left(-\frac{z_2+\sqrt az_0}{2\sqrt az_1}\right), \quad
D'_9\mapsto\rho^*\left(\frac{z_2-\sqrt az_0}{2\sqrt az_1}\right).
\end{equation}
Then the $\Kbar$-algebra 
\begin{equation}\label{eq:cox_ring_id}
\Kbar[\eta_1,\dots,\eta_9]/(\eta_6\eta_9+\eta_7\eta_8+\eta_1\eta_3\eta_4^2\eta_5^3),
\end{equation}
endowed with the $\Pic(\tS_{\Kbar})$-grading induced by assigning degree 
$[D_i]$ to $\eta_i$ for $i\in\{1,\dots,9\}$, together with the character \eqref{eq:coxringcharacter}
gives a Cox ring of $\tS_{\Kbar}$ of type $\identity_{\Pic(\tS_{\Kbar})}$ in the sense of 
\cite[Proposition 4.4]{arXiv:1408.5358}.

We observe that the Galois action of $\Gal(\Kbar/K)$ on $\tS_{\Kbar}$
exchanges $D_6$ with $D_7$ and $[D_8]$ with $[D_9]$, while $D_1,\dots,D_5$ are
fixed.  Hence $\Pic(\tS)=\Pic(\tS_{\Kbar})^{\Gal(\overline K/K)}\cong\ZZ^5$
with basis $\ell_0,\dots,\ell_3,\ell_4+\ell_5$.

We compute a Cox ring of $\tS$ of type $\Pic(\tS)\hookrightarrow\Pic(\tS_{\Kbar})$
in the sense of \cite{arXiv:1408.5358}.

\begin{lemma}
  The $\Kbar$-algebra
  \begin{equation*}
    \overline R\coloneqq \Kbar[\zeta_1,\dots,\zeta_8]/
    (\zeta_1\zeta_8+\zeta_7(\zeta_7+\zeta_2^2\zeta_3\zeta_4^3\zeta_6)),
  \end{equation*}
  where $\zeta_i$ have degrees $[E_i]$ for
  \begin{gather*}
    E_1\coloneqq D_6+D_7, \ E_2\coloneqq D_4, \ E_3\coloneqq D_3, \
    E_4\coloneqq D_5, \ E_5\coloneqq D_2, \\
    E_6\coloneqq D_1, \ E_7\coloneqq D_6+D_9, \ E_8\coloneqq D_8+D_9,
  \end{gather*}
  is a Cox ring of $\tS_{\Kbar}$ of type
  $\Pic(\tS)\hookrightarrow\Pic(\tS_{\Kbar})$.
\end{lemma}

\begin{proof}
  According to \cite[Corollary 2.10, Definition~3.12]{arXiv:1408.5358}, every
  Cox ring of $\tS_{\Kbar}$ of type $\Pic(\tS)\hookrightarrow\Pic(\tS_{\Kbar})$ is
  isomorphic to the sub-$\Kbar$-algebra of \eqref{eq:cox_ring_id} generated by the
  monomials $\prod_{j=1}^9\eta_j^{e_j}$ with $e_1,\dots,e_9\in\ZZ_{\geq0}$
  such that $\sum_{i=1}^9e_j[D_j]\in\Pic(\tS)$, that is,
  \begin{gather*}
    \zeta_6\coloneqq \eta_1,\ \zeta_5\coloneqq \eta_2,\ \zeta_3\coloneqq
    \eta_3,\ \zeta_2\coloneqq \eta_4,\ \zeta_4\coloneqq \eta_5,\\
    \zeta_1\coloneqq \eta_6\eta_7, \ \zeta_8\coloneqq \eta_8\eta_9, \
    \zeta_7\coloneqq \eta_6\eta_9,\
    \zeta_9\coloneqq \eta_7\eta_8.
  \end{gather*}
  The relation in \eqref{eq:cox_ring_id} implies
  $\zeta_9=-\zeta_7-\zeta_2^2\zeta_3\zeta_4^3\zeta_6$ and
  $\zeta_1\zeta_8+\zeta_7(\zeta_7+\zeta_2^2\zeta_3\zeta_4^3\zeta_6)=0$.  The
  definition of the $\zeta_i$ gives a surjective map from $\overline R$ onto a
  Cox ring of $\tS_{\Kbar}$ of type
  $\Pic(\tS)\hookrightarrow\Pic(\tS_{\Kbar})$.  Since the relation defining
  $\overline R$ is irreducible and
  $\dim \overline R=\dim S_{\Kbar}+\rk \Pic(\tS)$, we conclude that
  $\overline R$ is a Cox ring of $\tS_{\Kbar}$ of the desired type.
\end{proof}

\begin{lemma}\label{lem:cox_ring_injective_type}
  Every Cox ring of $\tS$ of type $\Pic(\tS)\hookrightarrow\Pic(\tS_{\Kbar})$ is
  isomorphic to the $K$-algebra
  \begin{equation*}
    R\coloneqq K[\xi_1,\dots,\xi_8]/(\xi_1\xi_8+\xi_7^2-a\xi_2^4\xi_3^2\xi_4^6\xi_6^2),
  \end{equation*}
  with $\xi_i$ homogeneous of degree $E_i$ for all $i\in\{1,\dots,8\}$.
\end{lemma}

\begin{proof}
  We observe that the action of $\Gal(K(\sqrt a)/K)$ on \eqref{eq:cox_ring_id}
  that fixes $\eta_1,\dots,\eta_5$, exchanges $\eta_6$ with $\eta_7$ and
  $\eta_8$ with $\eta_9$ turns \eqref{eq:cox_ring_id} and $\overline R$ into
  $\Gal(\Kbar/K)$-equivariant Cox rings of $\tS_{\Kbar}$ as in
  \cite[Definition 3.1]{arXiv:1408.5358}.  Let
  $\zeta_9\coloneqq -\zeta_7-\zeta_2^2\zeta_3\zeta_4^3\zeta_6$ as element of
  $\overline R$.  We observe that the elements $\zeta_7$ and $\zeta_9$ are
  exchanged by the Galois action, while all the other $\zeta_i$ are fixed. Let
  $\xi_i\coloneqq \zeta_i$ for $i\in\{1,\dots,6\}$ and consider the following
  change of variables:
  \begin{equation*}
    \xi_8\coloneqq 4a\zeta_8,\  \xi_9\coloneqq \zeta_7+\zeta_9,\
    \xi_7\coloneqq \sqrt a(\zeta_7-\zeta_9).
  \end{equation*}
  Then 
  \begin{equation*}
    \zeta_7=\frac{\xi_9}2+\frac{\xi_7}{2\sqrt a}, \
    \zeta_9=\frac{\xi_9}2-\frac{\xi_7}{2\sqrt a},
  \end{equation*}
  and $\xi_9=-\xi_2^2\xi_3\xi_4^3\xi_6$.
  So $4a\zeta_7\zeta_9=a\xi_9^2-\xi_7^2=a\xi_2^4\xi_3^2\xi_4^6\xi_6^2-\xi_7^2$. 
  
  Thus,
  \begin{equation*}
    \overline R=\Kbar[\xi_1,\dots,\xi_8]/(\xi_1\xi_8-a\xi_2^4\xi_3^2\xi_4^6\xi_6^2+\xi_7^2)
  \end{equation*}
  has $\Gal(\Kbar/K)$-invariant generators and relation. Hence it descend to a
  Cox ring of $\tS$ of type $\Pic(\tS)\hookrightarrow\Pic(\tS_{\Kbar})$.  Up to
  isomorphism, there is only one Cox ring of $\tS$ of type
  $\Pic(\tS)\hookrightarrow\Pic(\tS_{\Kbar})$ because the Galois action on
  $\Pic(\tS)$ is trivial (cf.~the comment after \cite[Proposition 3.7]{arXiv:1408.5358}).
\end{proof}

\subsection{A torsor over $\Kbar$}

The next lemma extends \cite[Remark~6]{MR2783385} to the Cox rings introduced in \cite{arXiv:1408.5358}.

\begin{lemma}\label{lem:bourqui}
  Let $X$ be a smooth surface over an algebraically closed field $k$. Let $\widetilde R$
  be a finitely generated Cox ring of $X$ of type $\lambda\colon M\to\Pic(X)$ such
  that $\lambda(M)$ contains an ample divisor class. Let $\eta_1,\dots,\eta_N$
  be homogeneous generators of $\widetilde R$. For every $i\in\{1,\dots,N\}$, let $D_i$
  be the divisor on $X$ defined by $\eta_i$. Then the open subset of
  $\Spec \widetilde R$ defined by $\prod_{j\notin J}\eta_j\neq0$ for all
  $J\subseteq\{1,\dots,N\}$ such that $\bigcap_{j\in J}D_j\neq\emptyset$ is a
  torsor of $X$ of type $\lambda$.
\end{lemma}

\begin{proof}
  Let $D$ be a very ample divisor on $X$ such that $[D]\in\lambda(M)$.  Let
  $Y=\Spec \widetilde R\smallsetminus V(\widetilde R_{[D]})$, where
  $\widetilde R_{[D]}$ is the degree-$[D]$-part of $\widetilde R$. Then $Y$ is
  a torsor of $X$ of type $\lambda$ by \cite[Corollary 4.3]{arXiv:1408.5358}.
  Let $\mathfrak{J}$ be the set of subsets $J\subseteq\{1,\dots,N\}$ such that
  $\bigcap_{j\in J}D_j=\emptyset$.  Since for every $i\in\{1,\dots,N\}$ the
  preimage of $D_i$ under $Y\to X$ is $Y\cap V(\eta_i)$, a subset
  $J\subseteq\{1,\dots,N\}$ belongs to $\mathfrak{J}$ if and only if
  $\widetilde R_{[D]}$ is contained in the ideal $(\eta_j:j\in J)$.  Hence,
  $\widetilde R_{[D]}\subseteq\bigcap_{J\in\mathfrak{J}}(\eta_j:j\in J)$.  We
  observe that
  $V(\prod_{j\notin J}\eta_j:J\notin\mathfrak{J})\subseteq
  V(\bigcap_{J\in\mathfrak{J}}(\eta_j:j\in J))$, because for every point
  $x\in V(\prod_{j\notin J}\eta_j:J\notin\mathfrak{J})
  =\bigcap_{J\notin\mathfrak{J}}\bigcup_{j\notin J}V(\eta_j)$, the set
  $J_x\coloneqq\{j\in\{1,\dots,N\}:x\in V(\eta_j)\}$ belongs to
  $\mathfrak{J}$.  Hence,
  $Y\subseteq\Spec \widetilde R\smallsetminus V(\prod_{j\notin
    J}\eta_j:J\notin\mathfrak{J})$.  For the other inclusion let
  $x\in \Spec \widetilde R\smallsetminus V(\prod_{j\notin J}\eta_j)$ for some
  $J\subseteq\{1,\dots,N\}$, $J\notin\mathfrak{J}$.  Since
  $\widetilde R_{[D]}\nsubseteq(\eta_j:j\in J)$, and $\widetilde R_{[D]}$ is
  generated by monomials, there exists $s\in \widetilde R_{[D]}$ such that
  $s(x)\neq0$. Hence $x\in Y$.
\end{proof}

By Lemma \ref{lem:bourqui} and the description of the negative curves on
$\tS_{\Kbar}$ from \cite[\S3.4]{math.AG/0604194}, the closed subset of
$\Spec R$ defined by the monomials
\begin{equation}\label{eq:f_i}
  \mon_{1,5,7}, \ \mon_{1,7,8},  \ \mon_{2,3}, \ \mon_{2,4}, \ \mon_{3,5},  \ \mon_{4,6}, \ \mon_{6,7,8},
\end{equation}
where
$\mon_{i_1,\dots,i_s}\coloneqq \prod_{1\leq j\leq7,j\notin\{
  i_1,\dots,i_s\}}\xi_j$, is the complement of a torsor of $\tS$ of type
$\Pic(\tS)\hookrightarrow\Pic(\tS_{\Kbar})$.

\subsection{A torsor over $\OO_K$}\label{sec:torsor}

Let $\pi\colon \Y\to\tSS$ be the $\OO_K$-model of the
torsor $Y\to\tS$ defined by
\begin{equation*}
  \mathcal{R}\coloneqq \OO_K[\xi_1,\dots,\xi_8]/(\xi_1\xi_8+\xi_7^2-a\xi_2^4\xi_3^2\xi_4^6\xi_6^2)
\end{equation*}
and the monomials \eqref{eq:f_i} as in
Construction~\ref{hypot3}.

\begin{prop}\label{prop:proj_model}
  The $\OO_K$-morphism $\pi\colon\Y\to\tSS$ is a torsor
  under $\GmO^5$, and the $\OO_K$-scheme $\tSS$ is
  projective.
\end{prop}

\begin{proof}
  The first statement follows from Remark~\ref{rem:model_construction} because
  for each monomial in \eqref{eq:f_i} the degrees of the variables appearing
  in the monomial generate $\Pic(\tS)$.

  From \cite[Remark 4.2]{MR3552013}, the divisor class
  $A\coloneqq 9\ell_0-3\ell_1-2\ell_2-\ell_3-(\ell_4+\ell_5)$ is ample.  Let
  $C_{K}$ and $C_{\OO_K}$ be the ideals generated by $g \coloneqq
  \xi_1\xi_8+\xi_7^2-a\xi_2^4\xi_3^2\xi_4^6\xi_6^2$
  and the monomials of
  degree $A$ in $K[\xi_1,\dots,\xi_8]$ and in $\OO_K[\xi_1,\dots,\xi_8]$,
  respectively. Then the radicals of $C_{K}$ and $C_{\OO_K}$ are the radicals
  of the ideal generated by $g$ and the following monomials
  \begin{equation*}
    \mon_{1,2,3,8},\ \mon_{1,3,5}, \ \mon_{1,5,7}, \ \mon_{1,7,8},
    \ \mon_{2,3,7}, \ \mon_{2,4,8}, \ \mon_{3,5,7},  \ \mon_{4,6,8}, \ \mon_{6,7,8}
  \end{equation*}
  in $K[\xi_1,\dots,\xi_8]$ and in $\OO_K[\xi_1,\dots,\xi_8]$, respectively.

  The relations
  \begin{gather*}
    \mon_{1,2,3,8}\xi_7^2=-A_{2,3}+aA_{1,7,8}\xi_2^3\xi_3\xi_4^6\xi_6^2\xi_7,
    \ 
    \mon_{1,3,5}\xi_7^2=-A_{3,5}\xi_8+aA_{1,5}\xi_2^4\xi_3\xi_4^6\xi_6^2, \\
    \mon_{2,3,7}\xi_7^2=-A_{2,3}\xi_7+aA_{1,7,8}\xi_2^3\xi_3\xi_4^6\xi_6^2\xi_1\xi_8,
    \
    \mon_{2,4,8}\xi_7^2=-A_{2,4}\xi_1+aA_{6,7,8}\xi_2^3\xi_3^2\xi_4^5\xi_6^3\xi_7, \\
    \mon_{3,5,7}\xi_7^2=-A_{3,5}\xi_7+aA_{1,5,7}\xi_2^4\xi_3\xi_4^6\xi_6^2\xi_1,
    \
    \mon_{4,6,8}\xi_7^2=-A_{4,6}\xi_1+aA_{6,7,8}\xi_2^4\xi_3^2\xi_4^5\xi_6^2\xi_7,
  \end{gather*}
  show that $\sqrt{C_{K}}=\sqrt{C'_{K}}$ and
  $\sqrt{C_{\OO_K}}=\sqrt{C'_{\OO_K}}$, where $C'_{K}$ and $C'_{\OO_K}$ are
  the ideals generated by $g$ and the monomials \eqref{eq:f_i} in
  $K[\xi_1,\dots,\xi_8]$ and in $\OO_K[\xi_1,\dots,\xi_8]$, respectively.

  The  polynomials
  \begin{gather*}
    \xi_1\xi_8+\xi_7^2-a\xi_2^4\xi_3^2\xi_4^6\xi_6^2, \
    \xi_2\xi_3\xi_4\xi_6\xi_7^2-a\xi_2^5\xi_3^3\xi_4^7\xi_6^3, \
    \xi_2\xi_4\xi_6\xi_7^3-a\xi_2^5\xi_3^2\xi_4^7\xi_6^3\xi_7, \\
    A_{1,5,7},\ A_{1,7,8}, \ A_{6,7,8}, \ \xi_7A_{1,6,8}, \
    \xi_2\xi_3\xi_5\xi_7^3, \ \xi_3\xi_5\xi_6\xi_7^3, \
    \xi_4\xi_5\xi_6\xi_7^3,
  \end{gather*}
  generate the ideal $C'_{\OO_K}$, as
  \begin{gather*}
    A_{3,5}=\xi_2\xi_4\xi_6\xi_7 g-(\xi_2\xi_4\xi_6\xi_7^3-a\xi_2^5\xi_3^2\xi_4^7\xi_6^3\xi_7),\\
    A_{2,3}=\xi_4\xi_5\xi_6\xi_7 g -\xi_2\xi_4\xi_6\xi_7^3+a \xi_2^3\xi_3\xi_4^6\xi_6^2\xi_7 A_{1,7,8},\\
    A_{2,4}=\xi_3\xi_5\xi_6\xi_7 g-\xi_3\xi_5\xi_6\xi_7^3 + a\xi_2^2\xi_3^2\xi_4^5\xi_6\xi_7 A_{1,7,8},\\
    A_{4,6}=\xi_2\xi_3\xi_5\xi_7 g-\xi_2\xi_3\xi_5\xi_7^3+
    a\xi_2^4\xi_3^2\xi_4^5\xi_6^2\xi_7A_{1,7,8},
    \\
    \xi_2\xi_3\xi_4\xi_6\xi_7^2-a\xi_2^5\xi_3^3\xi_4^7\xi_6^3=\xi_2\xi_3\xi_4\xi_6 g-\xi_1A_{1,5,7},\\
    \xi_7A_{1,6,8}=\xi_2\xi_3\xi_4\xi_5
    g-\xi_8A_{6,7,8}+a\xi_2^4\xi_3^2\xi_4^6\xi_6A_{1,7,8},
  \end{gather*}
  and can be easily checked to form a Gr\"obner basis of $C'_{\OO_K}$ for the
  lexicographic monomial ordering with respect to $\xi_8>\dots>\xi_1$ via
  \cite[Algorithm 4.2.2]{MR1287608}.

  Since the leading coefficient of each polynomial in the Gr\"obner basis is
  invertible in $\OO_K$, $C'_{K}\cap\OO_K[\xi_1,\dots,\xi_8]=C'_{\OO_K}$ by
  \cite[Proposition 4.4.4]{MR1287608}. Hence
  $\sqrt{C_{K}}\cap\OO_K[\xi_1,\dots,\xi_8]=\sqrt{C_{\OO_K}}$, and the
  $\OO_K$-model $\tSS$ that we constructed above is projective by
  Proposition~\ref{prop:model_properties}\ref{it:model_smooth}.
\end{proof}

\section{Parameterization}\label{sec:parameterizationnew}

In this section, we parameterize the rational points on
$\tS$ by integral points on its torsors, we construct a
good fundamental domain for the action of $\Gm^5(\OO_K)$ and we perform a
M\"obius inversion.

\subsection{Passage to the split torsor}\label{sec:parameterization_concrete}

Under the identification $\Pic(\tS)\cong\ZZ^5$ given by the basis
$\ell_0,\dots,\ell_3,\ell_4+\ell_5$, the action of $\Gm^5(\OO_K)$ on
$Y(K)$ is given by
\begin{equation}\label{eq:action}
  \underline u*(a_1,\dots,a_8)=(\underline u^{m^{(1)}}a_1,\dots,\underline u^{m^{(8)}}a_8),
\end{equation} 
where 
\begin{gather*}
  m^{(1)}=(0,0,0,0,1),\ m^{(2)}=(0,0,1,-1,0),\ m^{(3)}=(0,1,-1,0,0),\\
  m^{(4)}=(0,0,0,1,0),\ m^{(5)}=(1,-1,0,0,-1),\ m^{(6)}=(1,-1,-1,-1,0),\\
  m^{(7)}=(1,0,0,0,0),\ m^{(8)}=(2,0,0,0,-1),
\end{gather*}
and $\underline u^{(m_0,\dots,m_4)}=\prod_{i=0}^4u_i^{m_i}$ for all
$\underline u=(u_0,\dots,u_4)\in(\OO_K^\times)^5$.

\begin{lemma}\label{lem:parameterization}
  For the surface $S$,
  \begin{equation*}
    N_{U,H}(B) = \frac{1}{|\mu_K|}
    \sum_{\cfrb = (\cfr_0,\cfr_1,\cfr_2,\cfr_3,\cfr_4) \in \Cs^5} |M_\cfrb(B)|,
  \end{equation*}
  where $M_\cfrb(B)$ is the set of all
  \begin{equation*}
    (a_1, \dots, a_8) \in (\OO_{1*}\times \dots \times \OO_{8*}) \cap \FF
  \end{equation*}
  with height condition 
  \begin{equation}
  \label{eq:height_condition}
    \prod_{v \mid \infty} \tN_v(\sigma_v(a_1,\dots,a_7)) \le u_\cfrb B,
  \end{equation}
  torsor equation
  \begin{equation}\label{eq:torsor}
    a_1a_8+a_7^2-aa_2^4a_3^2a_4^6a_6^2 = 0,
  \end{equation}
  coprimality conditions
  \begin{equation}\label{eq:coprimality}
    \afr_8+\afr_5 \!=\! \afr_7+\afr_2\afr_3\afr_4 \!=\! \afr_6+\afr_1\afr_2\afr_3\afr_5 \!=\!
    \afr_5+\afr_2\afr_4 \!=\! \afr_4+\afr_1\afr_3 \!=\! \afr_3+\afr_1 \!=\! \afr_2+\afr_1 \!=\! \OO_K.
  \end{equation}
   Here, 
   $\FF$ is a
  fundamental domain for the action of $U_K \times (\OO_K^\times)^4$
  on $(K^\times)^6 \times K^2$ described by \eqref{eq:action}, 
  $\OO_{j*} \coloneqq \OO_j^{\ne 0}$ for $j \in \{1, \dots, 6\}$ and
  $\OO_{j*} \coloneqq \OO_j$ for $j \in \{7,8\}$, where
  $\OO_j \coloneqq \cfrb^{m^{(j)}} = \prod_{i=0}^4\cfr_i^{m^{(j)}_i}$ 
  for all
  $j \in \{1, \dots, 8\}$. 
  For 
  $v \in \Omega_K$ and $(x_1,\dots,x_7) \in K_v^7$ with $x_1 \ne 0$, 
  \begin{equation}\label{eq:localheightconditions}
    \tN_v(x_1,\dots,x_7) \coloneqq \max\left.\begin{cases}
   \left|\frac{x_6(ax_2^4x_3^2x_4^6x_6^2-x_7^2)}{x_1}\right|_v, \left|x_2x_3x_4x_5x_6x_7\right|_v,\\
    \left|x_1^2x_2x_3^2x_5^3\right|_v, \left|x_2^3x_3^2x_4^4x_5x_6^2\right|_v,
    \left|x_1x_2^2x_3^2x_4^2x_5^2x_6\right|_v
	\end{cases}    
    \right\}.
  \end{equation}
  Finally,
   \begin{equation}\label{eq:u_c_def}
    u_\cfrb \coloneqq \N(\cfr_0^3\cfr_1^{-1}\cdots\cfr_4^{-1}),
  \end{equation}
  and $\afr_j \coloneqq a_j\OO_j^{-1}$ for all
  $j \in \{1, \dots, 8\}$.
\end{lemma}

\begin{proof}
The map $\psi\colon Y\to\tS\to S\subseteq\PP^4_K$ that sends $(\xi_1,\dots,\xi_8)$ to
\begin{equation}\label{eq:anticanonical_sections}
(\xi_6\xi_8, 
\xi_2\xi_3\xi_4\xi_5\xi_6\xi_7,
    \xi_1^2\xi_2\xi_3^2\xi_5^3, 
    \xi_2^3\xi_3^2\xi_4^4\xi_5\xi_6^2,
    \xi_1\xi_2^2\xi_3^2\xi_4^2\xi_5^2\xi_6)
\end{equation}
is the composition of the chosen anticanonical embedding of $S$, the
minimal desingularization of $S$ and the torsor over $\tS$ from Section \ref{sec:dPsurfaces}.
The complement of the lines of $\PP^4_{\Kbar}$ contained in $S$ is 
\begin{equation*}
U=S\smallsetminus\{
x_4=x_1x_2=x_2x_3=x_1^2-ax_3^2=0\},
\end{equation*}
so $\psi^{-1}(U)=\{\xi_1\xi_2\xi_3\xi_4\xi_5\xi_6\neq0\}\subseteq\Spec R$.
Since $U$ is contained in the smooth locus of $S$, Proposition
\ref{prop:proj_model} and \cite[Theorem 2.7]{MR3552013} give
\begin{equation*}
  U(K)=\bigsqcup_{\cfrb\in\Cs^5}{}_{\cfrb}\pi({}_{\cfrb}\Y(\OO_K)\cap\psi^{-1}(U(K))),
\end{equation*}
where ${}_{\cfrb}\pi\colon{}_{\cfrb}\Y\to\tSS$ is a
$\cfrb$-twist of $\pi$ and ${}_{\cfrb}\Y(\OO_K)\cap\psi^{-1}(U(K))$
is the set of
$\underline a=(a_1, \dots, a_8) \in (\OO_{1*}\times \dots \times \OO_{8*})$
that satisfy \eqref{eq:torsor} and \eqref{eq:coprimality} because the radical
of the ideal of $\mathcal{R}$ generated by \eqref{eq:f_i} equals the radical
of the ideal
\begin{equation*}
  (\xi_8+\xi_5 )(\xi_7+\xi_2\xi_3\xi_4 )( \xi_6+\xi_1\xi_2\xi_3\xi_5 )(
  \xi_5+\xi_2\xi_4 )( \xi_4+\xi_1\xi_3 )( \xi_3+\xi_1 )(\xi_2+\xi_1).
\end{equation*}
Since the monomials \eqref{eq:f_i} belong to the radical of the ideal
generated by the monomials defining $\psi$,
\begin{equation*}
  H(\psi(a_1,\dots,a_8))=u_\cfrb^{-1}\prod_{v \mid \infty}
  \tN_v(\sigma_v(a_1,\dots,a_7))
\end{equation*}
(where we eliminate $a_8$ using (\ref{eq:torsor})) for all
$(a_1,\dots, a_8)\in {}_{\cfrb}\Y(\OO_K)$.
\end{proof}

\subsection{The fundamental domain}\label{sec:fundamental_domain}
Here we choose an explicit fundamental domain for Lemma \ref{lem:parameterization}.
Define $\Sigma$, $\delta$, $l$, $F$, $F(\infty)$, $F(B)$, $\Fs_1$ as in \cite[\S
5]{MR3552013}.
We recall that $\Fs_1\subseteq K^\times$ is a fundamental domain for the multiplicative action of $\OO_K^{\times}$ such that
\begin{equation}\label{eq:fundamental_domain_norm}
 | N(a')|^{\frac{d_v}{d}} \ll |a'|_v \ll |N(a')|^{\frac{d_v}{d}}
\end{equation}
for every $a' \in \Fs_1$ and every $v\mid\infty$.

For $\ab = (a_1,\dots, a_4) \in (K^\times)^4$ and $(x_{5v},x_{6v},x_{7v})\in (K_v^\times)^2 \times K_v$, we define
\begin{equation*}
  \tN_v(\ab;x_{5v},x_{6v},x_{7v})\coloneqq \tN_v(\sigma_v(a_1),\dots,\sigma_v(a_4),x_{5v},x_{6v},x_{7v}).
\end{equation*}
For $B\in \RR_{>0}\cup\{\infty\}$, let  $S_F(\ab;B)$ be the set
\begin{equation*}
  \bigwhere{(x_{5v},x_{6v},x_{7v})_v \in \prod_{v\mid\infty}
    (K_v^\times)^2 \times K_v}
  {\frac{1}{3}(\log \tN_v(\ab;x_{5v},x_{6v},x_{7v}))_v \in F(B^{\frac1{3d}})},
\end{equation*}
and 
\[
\Fs_0(\ab; B)\coloneqq \{(a_5,a_6,a_7) \in (K^\times)^2 \times K  
: \sigma(a_5,a_6,a_7) \in S_F(\ab;B)\}.
\]

As in \cite[\S5]{MR3552013}, the set
\[
\Fs\coloneqq \{(a_1,\dots,a_7) \in (K^\times)^6 \times K :
  \ab \in \Fs_1^4,\ (a_5,a_6,a_7) \in \Fs_0(\ab;\infty)\}\times K
\]
is a fundamental domain for the action of
$U_K \times (\OO_K^\times)^4$ on $(K^\times)^6 \times K^2$.

Since $F(B)=F+(-\infty, \log B]\delta$, where
$\delta=(d_v)_{v\in\Omega_{\infty}}$ and $F\subseteq\RR^{\Omega_{\infty}}$ is
contained in the hyperplane where the sum of the coordinates vanish,
$(\sigma_v(\ab),x_{5v},x_{6v},x_{7v})_v$ satisfies the height condition
\eqref{eq:height_condition} for all $(x_{5v},x_{6v},x_{7v})_v \in S_F(\ab,u_\cfrb B)$.

As we will see in Section \ref{sec:first_summation_completion}, the
contribution of the first summation to the main term consists of the volume of
$S_F(\ab;B)$, which we compute here.

For the next result, recall the definition of $\omega_v(S)$ in
Theorem~\ref{thm:number_fields}.

\begin{lemma}\label{lem:volume_S_F}
  With respect to the usual Lebesgue measure
  on $\prod_{v\mid\infty} K_v^3 = \RR^{3d}$, for $B > 0$, we have
  \begin{equation*}
    \vol(S_F(\ab;B)) = \frac{1}{3} \cdot 2^{r_1} \cdot
    \left(\frac{\pi}{4}\right)^{r_2}\cdot R_K  \cdot
    \left(\prod_{v\mid\infty} \omega_v(S)\right) \cdot \frac{B}{|N(a_2a_3a_4)|}.
  \end{equation*}
\end{lemma}

\begin{proof}
  This is analogous to \cite[Lemma~5.1]{MR3552013}.
\end{proof}

\subsection{M\"obius inversion}\label{sec:moebius}
In this section, we deal with the coprimality conditions in
\eqref{eq:coprimality} that involve $a_5,\dots,a_8$. The remaining ones are
encoded in the arithmetic function $\vartheta_0$ defined below. We also
replace the torsor equation \eqref{eq:torsor} by a sum over congruence classes
for a parameter $\rho$ as in \eqref{eq:eta} and eliminate the variable $a_8$.

Recall the notation introduced in Lemma \ref{lem:parameterization}.
We write
\begin{equation*}
  \ab = (a_1,\dots,a_4),\qquad \OO_* = \OO_{1*}\times\dots\times\OO_{4*},
  \qquad \afrb = (\afr_1,\dots,\afr_4).
\end{equation*}
Let
\begin{equation*}
  \vartheta_0(\afrb)\coloneqq
  \begin{cases}
    1, & \afr_4+\afr_1\afr_3 = \afr_3+\afr_1 = \afr_2+\afr_1 = \OO_K,\\
    0, & \text{else.}
  \end{cases}
\end{equation*}
Let $\dfrb = (\dfr_{56},\dfr_{58},\dfr_5,\dfr_6,\dfr_7)$ be a 5-tuple of
nonzero ideals. Let
\begin{equation*}
  \mu_K(\dfrb)\coloneqq\mu_K(\dfr_{56})\mu_K(\dfr_{58})\mu_K(\dfr_5)\mu_K(\dfr_6)\mu_K(\dfr_7).
\end{equation*}
For fixed $\cfrb\in\Cs^5$, $\ab$, $\dfrb$ as above, 
let $\gfr\coloneqq \dfr_{58}\afr_1+a\OO_K$, and let 
$\gfr'$ be the unique ideal with
$v_\p(\gfr') = \lceil v_\p(\gfr)/2 \rceil$ for all prime ideals $\p$.
Define the fractional ideals
\begin{equation}\label{eq:def_bfr}
  \bfr_5 \coloneqq  \dfr_5(\dfr_{56}\cap\dfr_{58})\OO_5,\qquad
  \bfr_6 \coloneqq  \dfr_6\dfr_{56}\OO_6,\qquad 
  \bfr_7\coloneqq \afr_1\dfr_7\dfr_{58}\gfr^{-1}\gfr'\OO_7.
\end{equation}
For $\rho\in\OO_K$, let
$\Gs(\cfrb,\ab,\dfrb,\rho)$ be the additive subgroup of $K^3$
consisting of all $(a_5,a_6,a_7)$ with $a_5 \in \bfr_5$ and $a_6 \in \bfr_6$ and
$a_7 \in \gamma_7\cdot a_6+\bfr_7$ with
$\gamma_7\in \OO_K$ satisfying 
\begin{equation*}
  \congr{\gamma_7}{0}{\gfr'\dfr_7\cfr_1\cfr_2\cfr_3},\qquad
  \congr{\gamma_7}{\rho a_2^2a_3a_4^3}{\dfr_{58}\afr_1\gfr^{-1}\gfr'\cfr_1\cfr_2\cfr_3}.
\end{equation*}
Note that such a $\gamma_7$ exists if $\dfr_7+\dfr_{58}\afr_1=\OO_K$, as
$\dfr_7\cfr_1\cfr_2\cfr_3+\dfr_{58}\afr_1\cfr_1\cfr_2\cfr_3=\cfr_1\cfr_2\cfr_3$ divides $a_2^2a_3a_4^3$.

The following technical lemma is similar to \cite[Lemma~5.6]{MR3269462}.

\begin{lemma}\label{lem:congruence}
  Let $\afr$ be an ideal and $\ffr$ a nonzero fractional ideal of $\OO_K$. Let
  $y_1,y_2 \in \ffr$ with $y_1\ffr^{-1}+\afr = \OO_K$. Then $y_2/y_1$ is
  defined modulo $\afr$ and, for $x \in \OO_K$, we have
  \begin{equation*}
    \congr{xy_1}{y_2}{\afr\ffr} \qquad \Longleftrightarrow \qquad \congrfr{x}{y_2/y_1}{\afr}.
  \end{equation*}
\end{lemma}

\begin{proof}
  The quotient $y_2/y_1$ is invertible modulo $\afr$ since
  $v_\p(y_1) = v_\p(\ffr) \le v_\p(y_2)$ for every $\p \mid
  \afr$. Furthermore, $\congr{xy_1}{y_2}{\afr\ffr}$ if and only if
  $v_\p(x-y_2/y_1) \ge v_\p(\afr)-v_\p(y_1\ffr^{-1})$, which is equivalent to
  $\congrfr{x}{y_2/y_1}{\afr}$ by our assumptions.
\end{proof}

\begin{lemma}\label{lem:moebius}
  Let $\cfrb \in \Cs^5$ and $B \ge 0$.  With $\Fs_1,\Fs_0(\ab;u_\cfrb B)$ as
  in Section \ref{sec:fundamental_domain} and $\Gs(\cfrb,\ab,\dfrb,\rho)$
  defined above, we have
  \begin{equation*}
    |M_\cfrb(B)| = \sum_{\ab \in \Fs_1^4 \cap \OO_*} \vartheta_0(\afrb)
    \sum_{\dfrb : (\ref{eq:moebius_1}),(\ref{eq:moebius_2})} \mu_K(\dfrb)
    \sum_{\rho:(\ref{eq:rho})} |\Gs(\cfrb,\ab,\dfrb,\rho)\cap \Fs_0(\ab;u_\cfrb B)|,
  \end{equation*}
  where the second sum runs over 5-tuples of ideals $\dfrb$ satisfying
\begin{equation}\label{eq:moebius_1}
  \dfr_{56} + \afr_1\afr_2\afr_3\afr_4 = \dfr_{58} + \afr_2\afr_3\afr_4 = \OO_K,
  \quad \forall\p: 2\nmid v_\p(a) \implies v_\p(\dfr_{58}\afr_1) \le v_\p(a),
\end{equation}
\begin{equation}\label{eq:moebius_2}
  \dfr_5 \mid \afr_2\afr_4,\qquad \dfr_6 \mid \afr_1\afr_2\afr_3,
  \qquad \dfr_7 \mid \afr_2\afr_3\afr_4,
\end{equation}
and the third sum runs over 
\begin{equation}\label{eq:rho}
  \rho \pmod{\dfr_{58}\afr_1\gfr^{-1}\gfr'} :
  \ \rho\OO_K+\dfr_{58}\afr_1\gfr^{-1}\gfr'=\gfr',
  \ \congr{\rho^2}{a}{\dfr_{58}\afr_1}.
\end{equation}
\end{lemma}

\begin{proof}
  For fixed $\ab$ with $\vartheta_0(\afrb)=1$ and $B>0$, 
  let $\Fs_0\coloneqq \Fs_0(\ab;u_\cfrb B)$, and
  \begin{equation*}
    \tilde M = \tilde M(\cfrb,\ab,B)\coloneqq |\{(a_5, \dots, a_8) : (a_1, \dots, a_8) \in M_\cfrb(B)\}|.
  \end{equation*}
  We apply M\"obius inversions: For $\afr_5+\afr_8=\OO_K$, we
  introduce $\dfr_{58}$. This gives
  \begin{equation*}
    \tilde M = \sum_{\dfr_{58} \in \Is_K} \mu_K(\dfr_{58}) \left|\left\{
        \begin{aligned}
          &(a_5,\dots, a_8) \in
          ((\dfr_{58} \OO_5 \times \OO_6 \times \OO_7) \cap \Fs_0) \times \dfr_{58}\OO_8\\
          &(\ref{eq:torsor}),\ \afr_7+\afr_2\afr_3\afr_4 = \afr_6+\afr_1\afr_2\afr_3\afr_5 =
          \afr_5+\afr_2\afr_4 = \OO_K
        \end{aligned}
      \right\}\right|.
  \end{equation*}
  
  There is a one-to-one correspondence between such $(a_5,\dots,a_8)$
  satisfying the torsor equation and triples
  $(a_5,a_6,a_7) \in (\dfr_{58} \OO_5 \times \OO_6 \times \OO_7) \cap \Fs_0$ satisfying
  \begin{equation}\label{eq:congruence}
    \congr{a_7^2}{aa_2^4a_3^2a_4^6a_6^2}{a_1\dfr_{58} \OO_8 =
      \afr_1\dfr_{58}\cfr_0^2}.
  \end{equation}
  Furthermore, this congruence and the other coprimality conditions imply that
  $\dfr_{58}+\afr_2\afr_3\afr_4 = \OO_K$ and that every $\p$ with odd
  $v_\p(a)$ satisfies $v_\p(\dfr_{58}\afr_1) \le v_\p(a)$ (otherwise, we would
  have $\p \mid \afr_2\afr_3\afr_4\afr_6$, contradicting
  $\afr_1+\afr_2\afr_3\afr_4\afr_6=\OO_K$).  Hence $ \tilde M $ equals
  \begin{equation*}
    \!\!\! \sums{\dfr_{58} \in \Is_K\\\dfr_{58}+\afr_2\afr_3\afr_4 =
      \OO_K\\\forall\p: 2\nmid v_\p(a) \implies v_\p(\dfr_{58}\afr_1) \le
      v_\p(a)} \!\!\! \mu_K(\dfr_{58})
    \left|\left\{
        \begin{aligned}
          &(a_5,a_6,a_7) \in (\dfr_{58} \OO_5 \times \OO_6 \times \OO_7) \cap \Fs_0\\
          &(\ref{eq:congruence}),\ \afr_7+\afr_2\afr_3\afr_4 = \OO_K,\\
          &\afr_6+\afr_1\afr_2\afr_3\afr_5 = \afr_5+\afr_2\afr_4 = \OO_K
        \end{aligned}
      \right\}\right|.
  \end{equation*}
  We observe that $\dfr_{58}\afr_1+\afr_2\afr_3\afr_4\afr_6=\OO_K$, but
  $\gfr\coloneqq \dfr_{58}\afr_1+a\OO_K$ is not necessarily $\OO_K$. As above, let
  $\gfr'$ be the unique ideal with $v_\p(\gfr') = \lceil v_\p(\gfr)/2 \rceil$
  for all prime ideals $\p$. Then $a_7 \in \gfr'\OO_7$. Let
  $\qfr\coloneqq \dfr_{58}\afr_1\gfr^{-1}$.

  We claim: For every $a_7 \in \OO_7$ with (\ref{eq:congruence}) there is a
  unique $\rho \pmod{\qfr\gfr'}$ with
  \begin{equation}\label{eq:congruences_rho}
    \congr{a_7}{\rho a_2^2 a_3 a_4^3 a_6}{\qfr\gfr'\cfr_0},\qquad
    \congr{\rho^2}{a}{\qfr\gfr},\qquad \rho\OO_K+\qfr\gfr'=\gfr'.
  \end{equation}
  Conversely, if $a_7 \in \OO_7$ satisfies the first congruence in
  (\ref{eq:congruences_rho}) for some $\rho$ with the second congruence
  in (\ref{eq:congruences_rho}), then $a_7$ satisfies (\ref{eq:congruence}).

  For the proof of this claim, our first step is to observe: Every
  $\rho \in \OO_K$ satisfying the second congruence in
  (\ref{eq:congruences_rho}) also satisfies the last condition in
  (\ref{eq:congruences_rho}). To prove this observation, we must show that
  $\min\{v_\p(\rho),v_\p(\qfr\gfr')\} = v_\p(\gfr')$ for every prime ideal
  $\p$. Since $\gfr=\qfr\gfr+a\OO_K$, we have
  $v_\p(\gfr) = \min\{v_\p(\qfr\gfr),v_\p(a)\}$. In the case
  $v_\p(\gfr) = v_\p(\qfr\gfr) \le v_\p(a)$, we have $v_\p(\qfr)=0$, our
  congruence implies $v_\p(\rho) \ge v_\p(\gfr)/2$, and therefore
  $v_\p(\rho) \ge v_\p(\gfr')$, which proves the observation. In the case
  $v_\p(\gfr) = v_\p(a) < v_\p(\qfr\gfr)$, our congruence gives
  $2 \mid v_\p(a)$ and $v_\p(\rho) = v_\p(a)/2 = v_\p(\gfr')$, which proves
  the observation.

  We write $A=a_2^2a_3a_4^3a_6 \in \OO_2^2\OO_3\OO_4^3\OO_6=\cfr_0$ and
  $\Afr=A\cfr_0^{-1}=\afr_2^2\afr_3\afr_4^3\afr_6$ for simplicity. Now we
  prove one direction of our claim: Since $a_7 \in \OO_7=\cfr_0$ and
  $A \in \cfr_0$ with $A\cfr_0^{-1}+\qfr\gfr'=\Afr+\qfr\gfr'= \OO_K$, we can
  apply Lemma~\ref{lem:congruence} (with $\afr = \qfr\gfr'$, $\ffr = \cfr_0$,
  $y_1=A$, $y_2=a_7$, $x=\rho$) to see that there is a $\rho \in \OO_K$
  satisfying the first part of (\ref{eq:congruences_rho}), and that such a
  $\rho$ is unique modulo $\qfr\gfr'$. This congruence implies
  $\gfr'\cfr_0 \mid a_7-\rho A$; together with $\gfr'\cfr_0 \mid a_7$ and
  $\gfr'+\Afr=\OO_K$, this gives $\gfr' \mid \rho$.  By
  \cite[Lemma~5.7]{MR3269462} (with $\afr_1=\qfr$, $\afr_2=\gfr'\cfr_0$),
  $\congr{a_7^2}{\rho^2A^2}{\qfr\gfr'^2\cfr_0^2}$; since $\gfr \mid \gfr'^2$,
  this implies $\congr{a_7^2}{\rho^2A^2}{\qfr\gfr\cfr_0^2}$. With
  (\ref{eq:congruence}), we get
  $\congr{\rho^2A^2}{aA^2}{\qfr\gfr\cfr_0^2}$. By Lemma~\ref{lem:congruence}
  (with $\afr = \qfr\gfr$, $\ffr=\cfr_0^2$, $y_1=A^2$, $y_2=aA^2$,
  $x=\rho^2$), we have $\congr{\rho^2}{a}{\qfr\gfr}$, the second congruence in
  (\ref{eq:congruences_rho}).

  Now we prove the converse direction of the claim: Given such $a_7$ and
  $\rho$, the first congruence in (\ref{eq:congruences_rho}) together with
  \cite[Lemma~5.7]{MR3269462} (which we may apply with $\afr_1=\qfr$ and
  $\afr_2=\gfr'\cfr_0$ since the third part of (\ref{eq:congruences_rho})
  gives $\rho A \in \gfr'\cfr_0$ and the first congruence in
  (\ref{eq:congruences_rho}) then gives $a_7 \in \gfr'\cfr_0$) implies
  $\congr{a_7^2}{\rho^2A}{\qfr\gfr'^2\cfr_0}$ and hence
  $\congr{a_7^2}{\rho^2A}{\qfr\gfr\cfr_0}$. The second congruence in
  (\ref{eq:congruences_rho}) and Lemma~\ref{lem:congruence} (with
  $\afr = \qfr\gfr$, $\ffr = \cfr_0^2$, $y_1=A^2$, $y_2=aA^2$, $x=\rho^2$)
  give $\congr{\rho^2A^2}{aA^2}{\qfr\gfr\cfr_0^2}$. Therefore, $a_7$ satisfies
  (\ref{eq:congruence}).
  
  Our claim implies that
  \begin{align*}
    \tilde M =
    &\sums{\dfr_{58} \in \Is_K\\\dfr_{58}+\afr_2\afr_3\afr_4 = \OO_K\\\forall\p: 2\nmid v_\p(a) \implies v_\p(\dfr_{58}\afr_1) \le v_\p(a)} \mu_K(\dfr_{58})
    \sums{\rho \pmod{\qfr\gfr'}\\\rho\OO_K+\qfr\gfr'=\gfr'\\\congr{\rho^2}{a}{\qfr\gfr}} \\
    &\left|\left\{\begin{aligned}
          &(a_5,a_6,a_7) \in (\dfr_{58} \OO_5 \times \OO_6 \times \gfr'\OO_7) \cap \Fs_0\\
          &\congr{a_7}{\rho a_2^2 a_3 a_4^3 a_6}{\qfr\gfr'\cfr_0},\\
          &\afr_7+\afr_2\afr_3\afr_4 = \afr_6+\afr_1\afr_2\afr_3\afr_5 = \afr_5+\afr_2\afr_4 = \OO_K
        \end{aligned}\right\}\right|.
  \end{align*}
  Next, we remove $\afr_6+\afr_5=\OO_K$ by M\"obius inversion:
  \begin{align*}
    \tilde M = &\sums{\dfr_{56},\dfr_{58} \in \Is_K\\(\ref{eq:moebius_1})}
    \mu_K(\dfr_{56})\mu_K(\dfr_{58}) \sum_{\rho:(\ref{eq:rho})}\\
    &\left|\left\{\begin{aligned}
          &(a_5,a_6,a_7) \in
          ((\dfr_{56}\cap\dfr_{58}) \OO_5 \times \dfr_{56}\OO_6 \times \gfr'\OO_7) \cap \Fs_0\\
          &\congr{a_7}{\rho a_2^2 a_3 a_4^3 a_6}{\qfr\gfr'\cfr_0},\\
          &\afr_7+\afr_2\afr_3\afr_4 = \afr_6+\afr_1\afr_2\afr_3 = \afr_5+\afr_2\afr_4 = \OO_K
        \end{aligned}
      \right\}\right|.
  \end{align*}
  Here, we could restrict ourselves to $\dfr_{56}+\afr_1\afr_2\afr_3\afr_4=\OO_K$.
  
  Three more M\"obius inversions give
  \begin{align*}
    \tilde M = &\sums{\dfrb \in \Is_K^5\\(\ref{eq:moebius_1}),(\ref{eq:moebius_2})}
    \mu_K(\dfrb) \sum_{\rho:(\ref{eq:rho})}
    \left|\left\{
    \begin{aligned}
      &(a_5,a_6,a_7) \in \bfr_5 \times \bfr_6 \times \dfr_7\gfr'\OO_7) \cap \Fs_0\\
      &\congr{a_7}{\rho a_2^2 a_3 a_4^3 a_6}{\qfr\gfr'\cfr_0},\\
      &\afr_7+\afr_2\afr_3\afr_4 = \afr_6+\afr_1\afr_2\afr_3 = \afr_5+\afr_2\afr_4 = \OO_K
    \end{aligned}
        \right\}\right|.    
  \end{align*}

  Note that $a_7 \in \dfr_7\gfr'\OO_7$ and
  $\congr{a_7}{\rho a_2^2a_3a_4^3a_6}{\qfr\gfr'\cfr_0}$ are
  equivalent to
  \begin{equation*}
    \congr{a_7}{0}{\dfr_7\gfr'\cfr_0},\qquad
    \congr{a_7}{\rho a_2^2a_3a_4^3a_6}{\dfr_{58}\afr_1\gfr^{-1}\gfr'\cfr_0}
  \end{equation*}
  By our construction of $\gamma_7$, if $a_7 = \gamma_7a_6$, then
  this system of equivalences is satisfied, and this value for $a_7$ is
  unique modulo the lcm of the moduli. Hence it is equivalent to
  \begin{equation*}
    \congr{a_7}{\gamma_7a_6}{(\dfr_7\gfr'\cfr_0)\cap(\dfr_{58}\afr_1\gfr^{-1}\gfr'\cfr_0)
      = \dfr_7\dfr_{58}\afr_1\gfr^{-1}\gfr'\cfr_0}.\qedhere
  \end{equation*}
\end{proof}

\section{The first summation}\label{sec:first_summation}

In this section, we perform the summation over the variables $a_5,a_6,a_7$. We
first show that the region where $a_5, a_6$ are small does not contribute to
the main term; see Section~\ref{sec:small_conjugates}.  This yields a new
counting problem to which \cite{MR3264671} applies. We estimate the error term
in Section \ref{sec:volumes_of_projections} and the main term in Section
\ref{sec:first_summation_completion}, where the summation over the remaining
variables $a_1,\dots,a_4$ is translated into a summation over ideals
$\afr_1,\dots,\afr_4$.

\subsection{Summation of error terms}\label{sec:summation_lemma}

Lemma~\ref{lem:summation} will be used to bound all error terms in the first
summation when summed over the remaining variables $a_1,\dots,a_4$ and the
auxiliary variables $\dfr$ and $\rho$.

\begin{lemma}\label{lem:quadratic_condition}
  Let $\afr$ be a nonzero fractional ideal of $K$ and $b \in K$. Let $y_v \in
  K_v$ and $c_v > 0$ for all $v \mid \infty$. Let
  \begin{equation*}
    \Bfr \coloneqq  \{a' \in K : |\sigma_v(a')^2-y_v|_v \le c_v\text{ for all }v \mid \infty\}
  \end{equation*}
  Then
  \begin{equation*}
    |(b+\afr) \cap \Bfr| \ll \frac{1}{\N\afr} \left(\prod_{v \mid \infty} c_v\right)^{\frac{1}{2}}+1.
  \end{equation*}
\end{lemma}

\begin{proof}
  For real $v$, the condition $|\sigma_v(a')^2-y_v|_v \le c_v$ describes one or
  two intervals in $\RR$ of length $\ll c_v^{1/2}$ (see
  \cite[Lemma~5.1(a)]{MR2520770}). For complex $v$, this condition describes a
  subset of $\CC$ contained in one or
  two balls of radius $\ll c_v^{1/2}$ as in \cite[Lemma~3.2]{MR3269462}.  Arguing
  as in \cite[Lemma~3.3]{MR3269462} and \cite[Lemma~7.1]{MR3552013} proves the
  result.
\end{proof}

\begin{lemma}\label{lem:summation}
  Let $\cfrb \in \Cs^5$, $\epsilon,\alpha_5,\alpha_6,\alpha_7 > 0$,
  $\alpha \in [0,1]$ with $\alpha\alpha_5+\alpha_6>1$ and
  $(1-\alpha)\alpha_5+\alpha_7>1$,
  $(e_1, \dots, e_4) \in \ZZ^4$ with $e_i > 0$ for some $i \in \{1, \dots, 4\}$ and
  $\beta \in \RRnz$. Consider the conditions
  \begin{equation}\label{eq:general_height}
    |N(a_1^{e_1}\cdots a_4^{e_4})|^{\sgn\beta} \ll B^{\sgn\beta},\qquad
    |N(a_1)| \ll B,\dots, |N(a_4)| \ll B.
  \end{equation}
  Then, for $B\ge 3$, we have
  \begin{align*}
    &\sums{\ab \in \Fs_1^4 \cap \OO_*\\(\ref{eq:general_height})}
    \sums{\dfrb\\(\ref{eq:moebius_1}),(\ref{eq:moebius_2})}
    \frac{|\mu_K(\dfrb)|}{\N(\bfr_5^{\alpha_5}\bfr_6^{\alpha_6}\bfr_7^{\alpha_7})}
    \sums{\rho\\(\ref{eq:rho})} \frac{B}{|N(a_1^{1-\alpha_7}a_2a_3a_4)|}
    \left(\frac{B}{|N(a_1^{e_1}\cdots a_4^{e_4})|}\right)^{-\beta}\\
    &\ll_\epsilon B(\log B)^{3+\epsilon},
  \end{align*}
  where the implicit constant also depends on
  $K,\alpha_5,\alpha_6,\alpha_7,\alpha,e_1,\dots,e_4,\beta$ and the
  implicit constants in (\ref{eq:general_height}).
\end{lemma}
 
\begin{proof}
  The proof is similar to \cite[Lemma~7.3]{MR3552013}. We have
  \begin{equation*}
    \N(\bfr_5^{\alpha_5}\bfr_6^{\alpha_6}\bfr_7^{\alpha_7})|N(a_1^{1-\alpha_7}a_2a_3a_4)|
    \!\gg\!
    \N(\dfr_5^{\alpha_5}\dfr_6^{\alpha_6}\dfr_7^{\alpha_7}\dfr_{56}^{\alpha\alpha_5+\alpha_6}
    \dfr_{58}^{(1-\alpha)\alpha_5+\alpha_7})|N(a_1a_2a_3a_4)|.
  \end{equation*}
  The sum over $\rho$ gives the factor
  \begin{equation*}
    \eta(\dfr_{58}\afr_1;a) \ll 2^{\omega_K(\dfr_{58})}\eta(\afr_1;a),
  \end{equation*}
  using Lemma~\ref{lem:eta_properties}. The sums over $\dfr_{56},\dfr_{58}$
  converge. The sums over $\dfr_5,\dfr_6,\dfr_7$ are
  $\ll (1+\epsilon/5)^{\omega_K(\afr_1\afr_2\afr_3\afr_4)}$.  We may replace the
  summation over $\ab$ by a summation over ideals $\afrb$. Therefore, it is
  enough to show that
  \begin{equation*}
    \sum_{\afrb \in \Is_K^4:(\ref{eq:general_height_ideal})}
    \frac{\eta(\afr_1;a)\cdot(1+\epsilon/5)^{\omega_K(\afr_1\afr_2\afr_3\afr_4)}B}
    {\N(\afr_1\cdots\afr_4)}
    \left(\frac{B}{\N(\afr_1^{e_1}\cdots\afr_4^{e_4})}\right)^{-\beta} \ll B(\log B)^{3+\epsilon}.
  \end{equation*}
  where
  \begin{equation}\label{eq:general_height_ideal}
    \N(\afr_1^{e_1}\cdots \afr_4^{e_4})^{\sgn\beta} \ll B^{\sgn\beta},
    \qquad \N(\afr_1) \ll B,\dots, \N(\afr_4) \ll B.
  \end{equation}
  Indeed, if $e_i > 0$, we sum over $\afr_i$ first using the first bound from
  (\ref{eq:general_height_ideal}) and then over the remaining $\afr_j$ using
  the remaining bounds from (\ref{eq:general_height_ideal}). Here, we combine
  \cite[Lemma~2.4]{MR3269462} with Lemma~\ref{lem:upper_bound_eta_omega_new}
  for the summation over $\afr_1$ and with \cite[Lemma~2.9]{MR3269462} for the
  summation over $\afr_2, \afr_3, \afr_4$.
\end{proof}

\subsection{Small conjugates}\label{sec:small_conjugates}

In Lemma \ref{lem:moebius}, we have $|N(a_i)| \ge \N \bfr_i$, for $i=5,6$,
since $a_i \ne 0$ and $a_i \in \bfr_i$ by definition of
$\Gs(\cfrb,\ab,\dfrb,\rho)$. We would like to strengthen this to the condition
$|a_i|_v \ge \N \bfr_i^{d_v/d}$ for all $v\mid\infty$.  We set
\begin{equation*}
  S_F^*(\cfrb,\ab,\dfrb;B)\coloneqq \{(x_{5v},x_{6v},x_{7v})_v \in S_F(\ab;B) :
  \forall i=5,6,\ \forall v: |x_{iv}|_v\ge \N \bfr_i^{d_v/d}\}
\end{equation*}
and
\begin{equation*}
  \Fs_0^*(\cfrb, \ab, \dfrb; u_\cfrb B) \coloneqq 
  \{(a_5,a_6,a_7) \in (K^\times)^2 \times K
  : \sigma(a_5,a_6,a_7) \in S_F^*(\cfrb,\ab,\dfrb; u_\cfrb B)\}.
\end{equation*}
We show that we can replace $\Fs_0(\ab;u_\cfrb B)$ by
$\Fs_0^*(\cfrb, \ab, \dfrb; u_\cfrb B)$ in Lemma~\ref{lem:moebius} and that we
may restrict the summation over $\ab$ to those satisfying the new height
condition
\begin{equation}\label{eq:height_new}
  N(a_1^{-1}a_2^4a_3^2a_4^6) \le \N(\OO_1^{-1}\OO_2^4\OO_3^2\OO_4^6) B.
\end{equation}

\begin{lemma}\label{lem:height_new}
  Let $\cfrb \in \Cs^5$ and $\epsilon > 0$. Then, for $B \ge 3$,
  \begin{align*}
    |M_\cfrb(B)| = {}&\sums{\ab \in \Fs_1^4 \cap \OO_*\\(\ref{eq:height_new})} \vartheta_0(\afrb)
    \sums{\dfrb\\(\ref{eq:moebius_1}),(\ref{eq:moebius_2})} \mu_K(\dfrb)
    \sums{\rho\\(\ref{eq:rho})}
    |\Gs(\cfrb,\ab,\dfrb,\rho)\cap \Fs_0^*(\cfrb,\ab,\dfrb;u_\cfrb B)|\\
                    &+ O_\epsilon(B(\log B)^{3+\epsilon}).
  \end{align*}
\end{lemma}

\begin{proof}
  We show that we can replace $\Fs_0(\ab;u_\cfrb B)$ by
  $\Fs_0^*(\cfrb, \ab, \dfrb; u_\cfrb B)$ as in \cite[Lemmas~7.4,
  7.5]{MR3552013}, using Lemma~\ref{lem:quadratic_condition} instead of
  \cite[Lemma~7.1]{MR3552013} to prove the analogue of
  \cite[(7.13)]{MR3552013}, and Lemma~\ref{lem:summation} instead of
  \cite[Lemma~7.3]{MR3552013} to show that the error terms are sufficiently
  small.

  We show that we can introduce the condition (\ref{eq:height_new}) as
  in \cite[Lemma~7.6]{MR3552013}, again using
  Lemma~\ref{lem:summation} instead of \cite[Lemma~7.3]{MR3552013}.
\end{proof}

\subsection{Definability in an o-minimal structure}\label{sec:o-minimal}

In order to apply \cite{MR3264671} as in \cite{MR3552013}, we interpret
$|\Gs(\cfrb,\ab,\dfrb,\rho)\cap \Fs_0^*(\cfrb,\ab,\dfrb;u_\cfrb B)|$ in
Lemma~\ref{lem:height_new} as the number of lattice points in a set that is
definable in an o-minimal structure. Here and in the following, we write for
simplicity
\begin{equation*}
  \Gs\coloneqq \Gs(\cfrb,\ab,\dfrb,\rho),\quad
  \Fs_0^*\coloneqq \Fs_0^*(\cfrb, \ab, \dfrb; u_\cfrb B),\quad
  S_F^*\coloneqq S_F^*(\cfrb, \ab,\dfrb;u_\cfrb B).
\end{equation*}
We define the $\RR$-linear map
\begin{equation*}
  \tau\colon \prod_{v\mid\infty} K_v^3 \to \prod_{v\mid\infty} K_v^3
\end{equation*}
by
\begin{equation*}
  \tau((x_{5v},x_{6v},x_{7v})_v) \coloneqq 
  \left(\frac{x_{5v}}{\N\bfr_5^{\frac{1}{d}}},\frac{x_{6v}}{\N\bfr_6^{\frac{1}{d}}},\frac{x_{7v}}
    {\N\bfr_7^{\frac{1}{d}}}\right)_v.
\end{equation*}
Let
$\Lambda=\Lambda(\cfrb,\ab,\dfrb,\rho)\coloneqq \tau(\sigma(\Gs))$. Since
$\sigma \colon K^3 \to \prod_{v\mid\infty} K_v^3 \cong \RR^{3d}$ is
an embedding, we have
\begin{equation}\label{eq:lemma_8.1}
  |\Gs\cap\Fs_0^*| = |\sigma(\Gs) \cap S_F^*| = |\Lambda \cap \tau(S_F^*)|.
\end{equation}

Because of (\ref{eq:lemma_8.1}), we want to estimate
$|\Lambda \cap \tau(S_F^*)|$. For this, we would like to use
\cite[Theorem~1.3]{MR3264671}.

Let $Z$ be the set of all
\begin{equation*}
(\beta,\beta_5,\beta_6,\beta_7,(x_{1v},\dots,x_{7v})_v) \in
\RR^4 \times \prod_{v\mid\infty} K_v^7
\end{equation*}
satisfying
\begin{align*}
  \beta,\beta_5,\beta_6,\beta_7 &> 0,\\
  |x_{1v}|_v,\dots,|x_{4v}|_v & > 0 \text{ for all } v\mid\infty,\\
  |x_{5v}|_v,|x_{6v}|_v & \geq 1 \text{ for all } v\mid\infty,\\
  (\tN_v(x_{1v}, \dots, x_{4v}, \beta_5 x_{5v}, \beta_6 x_{6v}, \beta_7 x_{7v})^{1/3})_v
                                &\in \exp(F(\beta^{\frac{1}{3d}})),
\end{align*}
where $\exp \colon \RR^{\Omega_\infty} \to \RR_{>0}^{\Omega_\infty}$ is the
coordinate-wise exponential function.

For
\begin{equation*}
  T = (\beta,\beta_5,\beta_6,\beta_7,(x_{1v},\dots,x_{4v})_v) \in
  \RR^4 \times \prod_{v\mid\infty} K_v^4,
\end{equation*}
we define the fiber
\begin{equation*}
  Z_T \coloneqq  \{(x_{5v},x_{6v},x_{7v})_v \in \prod_{v\mid\infty} K_v^3
  : (\beta,\beta_5,\beta_6,\beta_7,(x_{1v},\dots,x_{7v})_v) \in Z\}.
\end{equation*}
Then $\tau(S_F^*) = Z_T$ for
$T\coloneqq (u_\cfrb B, \N\bfr_5^{1/d}, \N\bfr_6^{1/d},
\N\bfr_7^{1/d},(\sigma_v(a_1),\dots,\sigma_v(a_4))_v)$.

\begin{lemma}\label{lem:o-minimal}
  The set $Z \subset \RR^{4+7d}$ is definable in the o-minimal
  structure $\RR_{\exp} = \langle \RR; <,+,\cdot,-,\exp\rangle$. The
  fibers $Z_T$ are bounded.
\end{lemma}

\begin{proof}
  This analogous to \cite[Lemma~9.1]{MR3552013}.
\end{proof}

\subsection{Volumes of projections}\label{sec:volumes_of_projections}

For every coordinate subspace $W$ of
$\prod_{v\mid\infty} K_v^3 = \RR^{3d}$, let
$V_W = V_W(\cfrb,\ab,\dfrb;u_\cfrb B)$ be the
$(\dim W)$-dimensional volume of the orthogonal projection of
$\tau(S_F^*)$ to $W$.

\begin{lemma}\label{lem:vol_projections}
  For every $\cfrb,\ab,\dfrb,\rho$ as in Lemma~\ref{lem:height_new}, we have
  \begin{equation*}
    |\Gs \cap \Fs_0^*| = \frac{2^{3r_2} \vol S_F^*}{|\Delta_K|^{3/2}\N(\bfr_5\bfr_6\bfr_7)}
    + O\left(\sum_W V_W\right),
  \end{equation*}
  where $W$ runs over all proper coordinate subspaces of $\RR^{3d}$.
\end{lemma}

\begin{proof}
  As in \cite[Lemma~8.2]{MR3552013},
  the set $\Lambda$ is a lattice of rank $3d$ and determinant
  $(2^{-r_2}\sqrt{|\Delta_K|})^3$, and Minkowski's first successive
  minimum $\lambda_1$ of this lattice with respect to the unit ball in
  $\RR^{3d}$ is at least $1$.

  Now we argue as in \cite[Lemma~10.1]{MR3552013}. We start with
  (\ref{eq:lemma_8.1}) and apply \cite[Theorem~1.3]{MR3264671}, which shows
  that
  \begin{equation*}
    |\Gs \cap \Fs_0^*| = |\Lambda \cap \tau(S_F^*)| =
    \frac{\vol(\tau(S_F^*))}{\det\Lambda} + O\left(\sum_W V_W\right).
  \end{equation*}
  By the definition of $\tau$, we have
  $\vol(\tau(S_F^*)) = \vol(S_F^*)/\N(\bfr_5\bfr_6\bfr_7)$. 
\end{proof}

We want to estimate the $V_W$. 
By the definition of $F(\infty)$ and
of $S_F^*$, every $(x_{5v},x_{6v},x_{7v})_v \in S_F^*$ satisfies
\begin{equation*}
  |x_{5v}|_v \ge \N\bfr_5^{\frac{d_v}{d}},\qquad 
  |x_{6v}|_v \ge \N\bfr_6^{\frac{d_v}{d}},\qquad 
  \tN_v(\ab;x_{5v},x_{6v},x_{7v}) \ll B^{\frac{d_v}{d}}
\end{equation*}
for every $v\mid\infty$. In particular, with the notation
\begin{equation*}
  c_5\coloneqq \left(\frac{B}{|N(a_1^2a_2a_3^2)|}\right)^{\frac 1 3},\qquad
  c_6\coloneqq \left(\frac{B}{|N(a_2^3a_3^2a_4^4)|}\right)^{\frac 1 2},\qquad
  c_7\coloneqq (|N(a_1)|B)^{\frac 1 2},
\end{equation*}
and with (\ref{eq:fundamental_domain_norm}), it satisfies
\begin{align}
  \label{eq:bound_x5} \N\bfr_5^{\frac{d_v}{d}} \ll |x_{5v}|_v
  &\ll c_5^{\frac{d_v}{d}},\\
  \label{eq:bound_x6} \N\bfr_6^{\frac{d_v}{d}} \ll |x_{6v}|_v
  &\ll c_6^{\frac{d_v}{d}}|x_{5v}|^{-\frac 1 2},\\
  \label{eq:bound_x7} |x_{7v}\pm \sqrt{\sigma_v(aa_2^4a_3^2a_4^6a_6^2)}|_v
  &\ll c_7^{\frac{d_v}{d}}|x_{6v}|_v^{-\frac{1}{2}},\\
  \label{eq:bound_b7} \N\bfr_7^{\frac{d_v}{d}}
  &\ll c_7^{\frac{d_v}{d}}|x_{6v}|_v^{-\frac{1}{2}}.
\end{align}
Here, (\ref{eq:bound_x5}), (\ref{eq:bound_x6}) and (\ref{eq:bound_x7}) follow from the condition
\[
\tN_v(\ab;x_{5v},x_{6v},x_{7v})\ll B^{\frac{d_v}{d}}.
\]
In particular, (\ref{eq:bound_x7}) follows from the condition
\begin{equation}
\label{eq:local_height_condition}
  |(x_{7v}^2-\sigma_v(aa_2^4a_3^2a_4^6a_6^2))x_{6v}\sigma_v(a_1^{-1})|_v \ll
  B^{\frac{d_v}{d}}
\end{equation}
using \cite[Lemma~5.1(1)]{MR2520770} and \cite[Lemma~3.5(1)]{MR3269462}.
For (\ref{eq:bound_b7}), we compute
\begin{align*}
  |N(a_1^{-1})\N\bfr_7^2|^{\frac{d_v}{d}}|x_{6v}|_v &\ll
  |N(a_1)\N(\dfr_{7}\dfr_{58})^2|^{\frac{d_v}{d}}|x_{6v}|_v  \ll
    |N(a_1a_2^2a_3^2a_4^2)|^{\frac{d_v}{d}} |x_{5v}^2x_{6v}|_v\\
  &\ll |\sigma_v(a_1a_2^2a_3^2a_4^2)x_{5v}^2x_{6v}|_v \ll
  \tN_v(\ab;x_{5v},x_{6v},x_{7v}) \ll B^{\frac{d_v}{d}}
\end{align*}
using (\ref{eq:def_bfr}), (\ref{eq:moebius_2}), (\ref{eq:bound_x5}),
(\ref{eq:fundamental_domain_norm}), and (\ref{eq:local_height_condition}).

Define $\tau_v \colon K_v^3 \to K_v^3$ by
$(x_{5v},x_{6v},x_{7v}) \mapsto
(\N\bfr_5^{-1/d}x_{5v},\N\bfr_6^{-1/d}x_{6v},\N\bfr_7^{-1/d}x_{7v})$. Let
\begin{equation*}
  S_F^{(v)}\coloneqq \{(x_{5v},x_{6v},x_{7v}) \in K_v^3 :
  \text{(\ref{eq:bound_x5})--(\ref{eq:bound_b7}) hold}\}.
\end{equation*}
Then
\begin{equation*}
  \tau(S_F) \subset \prod_{v \mid \infty} \tau_v(S_F^{(v)}).
\end{equation*}
Let
\begin{align*}
  s_0 &\coloneqq  \frac{c_7c_6^{\frac 1 2}c_5^{\frac 3 4}}{\N(\bfr_5\bfr_6\bfr_7)} =
        \frac{B}{\N(\bfr_5\bfr_6\bfr_7)N(a_2a_3a_4)},\\
  s_5 &\coloneqq  \frac{c_7c_6^{\frac 1 2}c_5^{\frac 1 4}}{\N(\bfr_5^{\frac 1 2}\bfr_6\bfr_7)} =
        \frac{B}{\N(\bfr_5^{\frac 1 2}\bfr_6\bfr_7)N(a_2a_3a_4)}
        \left(\frac{B}{N(a_1^2a_2a_3^2)}\right)^{-\frac{1}{6}},\\
  s_6 &\coloneqq  \frac{c_7c_6^{\frac 1 4}c_5^{\frac 7 8}}{\N(\bfr_5\bfr_6^{\frac 3 4}\bfr_7)} =
        \frac{B}{\N(\bfr_5\bfr_6^{\frac 3 4}\bfr_7)N(a_2a_3a_4)}
        \left(\frac{B}{N(a_1^{-1}a_2^4a_3^2a_4^6)}\right)^{-\frac{1}{12}},\\
  s_7 &\coloneqq  \frac{c_7^{\frac 1 2}c_6^{\frac 3 4}c_5^{\frac 5
        8}}{\N(\bfr_5\bfr_6\bfr_7^{\frac 1 2})} =
        \frac{B}{\N(\bfr_5\bfr_6\bfr_7^{\frac 1 2}) N(a_1^{\frac 1 2}a_2a_3a_4)}
        \left(\frac{B}{N(a_1^2a_2^{-2}a_3^{-1}a_4^{-3})}\right)^{-\frac{1}{6}}.
\end{align*}

\begin{lemma}\label{lem:vol_projections_estimates}
  Let $v\mid\infty$. For $P = (p_5,p_6,p_7) \in \{0,\dots,d_v\}^3$,
  let $V_P$ be the volume of the orthogonal projection of $\tau_v(S_F^{(v)})$
  to
  \begin{equation*}
    \RR^{3d_v-p_5-p_6-p_7}\cong \{\text{$f_{p_j}(x_j)=0$ for all $j \in \{5,6,7\}$ with
      $p_j>0$}\} \subset \RR^{2d_v}
  \end{equation*}
  where
  \begin{equation*}
    f_p(x) =
    \begin{cases}
      x, & p=1,\ v\text{ real},\\
      \Re x\text{ or }\Im x, & p=1,\ v\text{ complex},\\
      x , & p=2,\ v\text{ complex}.
    \end{cases}
  \end{equation*}
  Then
  \begin{equation*}
    V_P \ll
    \begin{cases}
      s_0^{d_v/d}, & p_5=p_6=p_7=0,\\
      s_5^{d_v/d}, & p_5 \ne 0,\ p_6=p_7=0,\\
      s_6^{d_v/d}, & p_6 \ne 0,\ p_7=0,\\
      s_7^{d_v/d}, & p_7 \ne 0.
    \end{cases}
  \end{equation*}
\end{lemma}

\begin{proof}
  The computations are exactly the same as \cite[Lemmas~10.2,
  10.3]{MR3552013}, except that the indices $5,6,7$ are shifted by one and
  that we have condition~(\ref{eq:bound_x7}) instead of
  $|x_{7v}| \ll c_7^{d_v/d}|x_{6v}|_v^{-1/2}$. But either condition is
  sufficient to show that
  \begin{equation*}
    \int \dd x_7 \ll \frac{c_7^{d_v/d}}{|x_{6v}|_v^{1/2}},
  \end{equation*}
  and they are not used elsewhere.
\end{proof}

\begin{lemma}\label{lem:V_W_small}
  Let $\cfrb \in \Cs^5$. Let $W$ be a proper coordinate subspace of
  $\prod_{v \mid \infty} K_v^3 = \RR^{3d}$. Let $V_W$ be the
  $\dim(W)$-dimensional volume of the orthogonal projection of
  $\tau(S_F^*)$ to $W$. Then, for $\epsilon > 0$, we have
  \begin{equation*}
    \sum_{\ab \in \Fs_1^4\cap\OO_*:(\ref{eq:height_new})}
    \vartheta_0(\ab) \sum_{\dfrb : (\ref{eq:moebius_1}),(\ref{eq:moebius_2}),S_F^*
      \ne \emptyset} |\mu_K(\dfrb)| \sum_{\rho:(\ref{eq:rho})} V_W
    \ll_\epsilon B(\log B)^{4-\frac 1 d+\epsilon}
  \end{equation*}
\end{lemma}

\begin{proof}
  For $j=0,5,6,7$, we estimate
  \begin{equation*}
    \Sigma_j\coloneqq \sum_{\ab} \sum_{\dfrb} |\mu_K(\dfrb)| \sum_{\rho} s_j.
  \end{equation*}
  We have $\Sigma_0 \ll B(\log B)^{4+\epsilon}$. Indeed, we argue similarly as
  in Lemma~\ref{lem:summation}, using
  Lemma~\ref{lem:upper_bound_eta_omega_new} for the summation over $\afr_1$.
  For all other $j=5,6,7$, we have $\Sigma_P \ll B(\log B)^{3+\epsilon}$
  because of Lemma~\ref{lem:summation} with the third height condition for
  $j=5,7$ and with (\ref{eq:height_new}) for $j=6$.
  
  Finally, we argue as in \cite[Lemma~10.4]{MR3552013}, applying H\"older's inequality.
\end{proof}

\subsection{Completion of the first summation}\label{sec:first_summation_completion}

We show that $S_F(\ab;u_\cfrb B)$ and $S_F^*$ have roughly the same
volume. Here, we use the condition
\begin{equation}\label{eq:height_1234}
  N(a_1^2a_2a_3^2) \le \N(\OO_1^2\OO_2\OO_3^2)B.
\end{equation}

\begin{lemma}\label{lem:S_F-S_F*}
  Let $\cfrb \in \Cs^5$ and $\epsilon>0$. Then
  \begin{equation*}
    \sums{\ab \in \Fs_1^4\cap\OO_*\\(\ref{eq:height_new}),(\ref{eq:height_1234})}
    \sums{\dfrb\\(\ref{eq:moebius_1}),(\ref{eq:moebius_2})} \!\!\!|\mu_K(\dfrb)|
    \sums{\rho\\(\ref{eq:rho})}
    \frac{\vol(S_F(\ab;u_\cfrb B) \smallsetminus S_F^*(\cfrb,\ab,\dfrb;u_\cfrb
      B))}{\N(\bfr_5\bfr_6\bfr_7)} \ll_\epsilon B(\log B)^{3+\epsilon}.
  \end{equation*}
\end{lemma}

\begin{proof}
  This is analogous to \cite[Lemma~11.1]{MR3552013}, using
  Lemma~\ref{lem:summation} instead of \cite[Lemma~7.3]{MR3552013}.
\end{proof}

\begin{prop}\label{prop:first_summation}
  For $\epsilon>0$ and $B \ge 3$, we have
  \begin{align*}
    N_{U,H}(B) 
    ={}&\frac{\rho_K}{3|\Delta_K|}
         \sums{\afrb \in \Is_K^4\\(\ref{eq:height_1234_ideal})}
    \frac{\vartheta_1(\afrb)B}{\N(\afr_1\afr_2\afr_3\afr_4)}
    \left(\prod_{v\mid\infty} \omega_v(S)\right)
    + O_\epsilon(B(\log B)^{4-\frac 1 d+\epsilon})
  \end{align*}
  with the height condition
  \begin{equation}\label{eq:height_1234_ideal}
    \N(\afr_1^2\afr_2\afr_3^2) \le B,\qquad 
    \N(\afr_1^{-1}\afr_2^4\afr_3^2\afr_4^6) \le B.
  \end{equation}
  and the notation
  \begin{equation*}
    \vartheta_1(\afrb) \coloneqq  \prod_{\p}
    \vartheta_{1,\p}(v_\p(\afr_1),\dots,v_\p(\afr_4))
  \end{equation*}
  with
  \begin{equation*}
    \vartheta_{1,\p}(\vv)\coloneqq 
    \begin{cases}
      \left(1-\frac{1}{\N\p}\right) \theta_{1,\p}(v_1),
      &\supp(\vv) = \emptyset, \{1\},\\
      \left(1-\frac{1}{\N\p}\right)^2, &\supp(\vv) = \{3\},\{4\},\\
      \left(1-\frac{1}{\N\p}\right)^3, &\supp(\vv) = \{2\},\{2,3\},\{2,4\},\\
      0, &\text{else}
    \end{cases}    
  \end{equation*}
  (using the notation $\supp(\vv)\coloneqq \{i  : v_i \ne 0\}$ from \cite[Definition~7.1]{MR2520770}) and
  \begin{equation*}
    \theta_{1,\p}(v_1)  \coloneqq 
    \begin{cases}
      1+\frac{1}{\N\p}-\frac{\eta(\p;a)}{\N\p^2}, &v_1 = 0,\\
      \frac{\eta(\p^{v_1};a)}{\N\p^{-\floor{\frac{\min\{v_1,v_\p(a)\}}{2}}}}
      -\frac{\eta(\p^{v_1+1};a)}{\N\p^{2-\floor{\frac{\min\{v_1+1,v_\p(a)\}}{2}}}}
      & v_1 > 0.
    \end{cases}
  \end{equation*}
\end{prop}

\begin{proof}
  By
  Lemmas~\ref{lem:parameterization},~\ref{lem:height_new},~\ref{lem:vol_projections},~\ref{lem:V_W_small},~\ref{lem:S_F-S_F*},~\ref{lem:volume_S_F},
  \begin{align*}
    N_{U,H}(B) = {}
    &\frac{2^{r_1}(2\pi)^{r_2} R_K}{3\cdot |\mu_K|\cdot |\Delta_K|^{3/2}}
      \left(\prod_{v\mid\infty} \omega_v(S)\right)\cdot
      \sum_{\cfrb \in \Cs^5}
      \sum_{\ab \in \Fs_1^4 \cap \OO_*:(\ref{eq:height_new}),(\ref{eq:height_1234})} \vartheta_0(\afrb)\cdot\\
    &\sum_{\dfrb:(\ref{eq:moebius_1}),(\ref{eq:moebius_2})}
      \sum_{\rho:(\ref{eq:rho})}
      \frac{\mu_K(\dfrb)u_\cfrb B}{\N(\bfr_5\bfr_6\bfr_7)N(a_2a_3a_4)}
      +O(B(\log B)^{4-\frac 1 d+\epsilon}).
  \end{align*}
  By definition (see (\ref{eq:u_c_def}), (\ref{eq:def_bfr})),
  \begin{equation*}
    \frac{\mu_K(\dfrb)u_\cfrb B}{\N(\bfr_5\bfr_6\bfr_7)N(a_2a_3a_4)}
    =\frac{B}{\N(\afr_1\cdots\afr_4)}\cdot
    \frac{\mu_K(\dfrb)}{\N(\dfr_5\dfr_6\dfr_7\dfr_{56}\dfr_{58}(\dfr_{56}\cap \dfr_{58})\gfr^{-1}\gfr')}.
  \end{equation*}
  In the summation over
  $\cfrb = (\cfr_0,\cfr_1,\cfr_2,\cfr_3,\cfr_4)\in \Cs^5$, the summands are
  independent of $\cfr_0$. Hence we can replace this sum by the factor $h_K$
  and a sum over $(\cfr_1,\dots,\cfr_4) \in \Cs^4$; we also use the notation
  $\rho_K$ (\ref{eq:rho_K_def}). As in the proof of
  \cite[Lemma~11.3]{MR3552013}, we transform the summation over
  $(\cfr_1,\dots,\cfr_4)\in \Cs^4$ and
  $(a_1,\dots,a_4) \in \Fs_1^4 \cap \OO_*$ into a summation over
  $(\afr_1,\dots,\afr_4) \in \Is_K^4$. The conditions (\ref{eq:height_new}),
  (\ref{eq:height_1234}) on $\ab$ correspond to (\ref{eq:height_1234_ideal})
  on $\afrb$.  It remains to show that
  \begin{equation*}
    \vartheta_1(\afrb)=\vartheta_0(\afrb) \sum_{\dfrb :
      (\ref{eq:moebius_1}),(\ref{eq:moebius_2})}
    \frac{\mu_K(\dfrb)}{\N(\dfr_5\dfr_6\dfr_7\dfr_{56}\dfr_{58}(\dfr_{56}\cap \dfr_{58})\gfr^{-1}\gfr')}
    \sum_{\rho:(\ref{eq:rho})} 1.
  \end{equation*}
  This is a straightforward computation, as in \cite[Lemma~11.2]{MR3552013}.
\end{proof}

\section{The remaining summations}\label{sec:remaining_summations}

The summation over the ideals $\afr_1,\dots,\afr_4$ appearing in
Proposition~\ref{prop:first_summation} is implemented in Section
\ref{sec:remaining_summations_completion} by adapting the techniques developed
in \cite[\S 7]{MR3269462}. In order to do so, we first need to estimate the
average size of the arithmetic function $\vartheta_1$.

\subsection{Averages of arithmetic functions}\label{sec:averages}

\begin{lemma}\label{lem:summation_theta_1}
  Let $\vartheta_1$ be as in Proposition~\ref{prop:first_summation}. Then, for all
  $\afr_2,\afr_3,\afr_4\in\Is_K$, 
  $x \ge 0$ and $\epsilon > 0$, we have
  \begin{equation*}
    \sum_{\substack{\afr_1\in\Is_K\\ \N\afr_1 \le x}}  \vartheta_1(\afr_1,\afr_2,\afr_3,\afr_4) = \rho_K \cdot \vartheta_2(\afr_2,\afr_3,\afr_4)\cdot x + O_{a,\epsilon}((1+\epsilon)^{\omega_K(\afr_2\afr_3\afr_4)}x^{1-\frac{1}{4d}+\epsilon}),
  \end{equation*}
  where
  \begin{equation*} 
    \vartheta_2(\afr_2,\afr_3,\afr_4)\coloneqq \prod_\p
    \vartheta_{2,\p}(v_\p(\afr_2),v_\p(\afr_3),v_\p(\afr_4))
  \end{equation*}
  and (with $r_a(\p)$ as defined in
  Theorem~\ref{thm:number_fields})
  \begin{equation*}
    \vartheta_{2,\p}(\vv) =
    \begin{cases}
      \left(1-\frac{1}{\N\p}\right)^2\left(1+\frac{2+r_a(\p)}{\N\p}\right), &
      \supp(\vv)=\emptyset,\\
       \left(1-\frac{1}{\N\p}\right)^3, &\supp(\vv) = \{3\},\{4\},\\
      \left(1-\frac{1}{\N\p}\right)^4, &\supp(\vv) = \{2\},\{2,3\},\{2,4\},\\
      0, &\text{else.}
    \end{cases}
  \end{equation*}
\end{lemma}

\begin{proof}
  We may assume $\afr_3+\afr_4=\OO_K$ since the statement is trivial
  otherwise. The statement clearly holds for $x<2$; hence we can assume
  $x \ge 2$.  For simplicity, we write
  \begin{equation*}
    \vartheta_1(\qfr)\coloneqq \vartheta_1(\qfr,\afr_2,\afr_3,\afr_4),\
    \vartheta_2\coloneqq \vartheta_2(\afr_2,\afr_3,\afr_4),\
    \vartheta_{2,\p}\coloneqq  \vartheta_{2,\p}(v_\p(\afr_2),v_\p(\afr_3),v_\p(\afr_4)).
  \end{equation*}
  
  Let $A_\p(n) =
  \vartheta_{1,\p}(n,v_\p(\afr_2),v_\p(\afr_3),v_\p(\afr_4))$. Then
  $\vartheta_1(\qfr) = \prod_\p A_\p(v_\p(\qfr))$.  Let
  $\bfr\coloneqq \afr_2\afr_3\afr_4$. We have
  \begin{equation*}
    A_\p(0) =
    \begin{cases}
      \left(1-\frac 1{\N\p}\right)\left(1+\frac
        1{\N\p}-\frac{\eta(\p;a)}{\N\p^2}\right), & \p \nmid \bfr,\\
      \left(1-\frac 1{\N\p}\right)^2, & \p \mid \afr_3\afr_4,\ \p \nmid \afr_2,\\
      \left(1-\frac 1{\N\p}\right)^3, & \p \mid \afr_2.
    \end{cases}
  \end{equation*}
  By Lemma~\ref{lem:eta_properties}, $1-\frac{3}{\N\p^2}<A_\p(0)<1$ for all
  $\p \nmid 2a$ and $0< A_\p(0) \le 1$ for all $\p \mid 2a$. In particular,
  $c_0\coloneqq \prod_\p A_\p(0)$ is a positive real number $\le 1$.  Furthermore, for
  $n > 0$, we have
  \begin{equation*}
    A_\p(n) =
    \begin{cases}
      0, & \p \mid \bfr,\\
      \left(1-\frac 1{\N\p}\right)\left(1-\frac
        1{\N\p^2}\right)\left(1+\legendre{a}{\p}\right), & \p \nmid 2a\bfr,\\
      \left(1-\frac 1{\N\p}\right)\theta_{1,\p}(n)
      & \p \mid 2a,\ \p \nmid \bfr.
    \end{cases}
  \end{equation*}
  
  For every $\p$ and $n > 0$, we define the multiplicative functions
  $B_0,B_1 \colon \Is_K \to \RR$ by
  \begin{equation*}
    B_0(\p^n) \coloneqq  \frac{A_\p(n)-A_\p(n-1)}{A_\p(0)}=
    \begin{cases}
      -1, & \p \mid \bfr,\ n=1,\\
      0, & \p \mid \bfr,\ n>1,\\
      \frac{\legendre{a}{\p}-\frac 1{\N\p}}{1+\frac
        1{\N\p}-\frac{1+\legendre{a}{\p}}{\N\p^2}}, & \p \nmid 2a\bfr,\ n=1,\\
      0, & \p \nmid 2a\bfr,\ n>1,\\
      \frac{\theta_{1,\p}(n)-\theta_{1,\p}(n-1)}{\theta_{1,\p}(0)}, &
      \p \mid 2a,\ \p \nmid \bfr,\ n \le
      v_\p(4a)+1,\\
      0, & \p \mid 2a,\ \p \nmid \bfr,\ n > v_\p(4a)+1.
    \end{cases}
  \end{equation*}
  (using Lemma~\ref{lem:eta_properties}) and
  \begin{equation*}
    B_1(\p^n) \coloneqq  \begin{cases}
      \frac{\legendre{a}{\p}-\frac 1{\N\p}}{1+\frac
        1{\N\p}-\frac{1+\legendre{a}{\p}}{\N\p^2}}, & n=1,\ \p \nmid 2a,\\
      0, & \text{else.}
    \end{cases}
  \end{equation*}
  For $\p \nmid 2a$, we observe that
  \begin{equation}\label{eq:B_1_bound}
    |B_1(\p)|\le 1,\qquad B_1(\p) = \legendre{a}{\p}+O\left(\frac{1}{\N\p}\right).
  \end{equation}
  
  We consider the formal Dirichlet series
  \begin{gather*}
    F(s)\coloneqq \sum_{\qfr \in \Is_K} \frac{\vartheta_1(\qfr)}{\N\qfr^s},\
    G(s)\coloneqq \sum_{\qfr \in \Is_K} \frac{(\vartheta_1*\mu_K)(\qfr)}{\N\qfr^s},\\
    G_0(s)\coloneqq \sum_{\qfr \in \Is_K} \frac{B_0(\qfr)}{\N\qfr^s},\
    G_1(s)\coloneqq \sum_{\qfr \in \Is_K} \frac{B_1(\qfr)}{\N\qfr^s},\
    L(s,\chi)\coloneqq \sum_{\qfr \in \Is_K} \frac{\chi(\qfr)}{\N\qfr^s},
  \end{gather*}
  with $\chi$ as in~(\ref{eq:def_chi}). We will discuss their convergence
  below.  Since $\vartheta_1 = (\vartheta_1*\mu_K) * 1$, we have
  \begin{equation*}
    F(s) = G(s)\zeta_K(s),
  \end{equation*}
  where $\zeta_K(s)$ is the Dedekind zeta function.
  
  A computation as in \cite[Lemma~2.2(1)]{MR3269462} and 
  \cite[Proposition~6.8(3)]{MR2520770} shows that 
  \begin{equation*}
    (\vartheta_1*\mu_K)(\qfr) = \prod_{\p \nmid \qfr} A_\p(0) \prod_{\p \mid
      \qfr} (A_\p(v_\p(\qfr))-A_\p(v_\p(\qfr)-1)) = c_0B_0(\qfr),
  \end{equation*}
  hence
  \begin{equation*}
    G(s) = c_0G_0(s).
  \end{equation*}
  Since $B_0,B_1$ are multiplicative, we may compare the Euler products
  of $G_0(s)$ and $G_1(s)$. Their Euler factors differ at most at the primes
  $\p \mid 2a\bfr$. More precisely, we have $G_0(s)=G_1(s)H_1(s)$ with
  \begin{equation*}
    H_1(s) \coloneqq  \prod_{\p \mid \bfr}
    \frac{1-\frac{1}{\N\p^s}}{1+\frac{B_1(\p)}{\N\p^s}} \prods{\p \mid 2a\\\p
      \nmid \bfr} \left(1+\frac{B_0(\p)}{\N\p^s}+\dots
      +\frac{B_0(\p^{v_\p(4a)+1})}{\N\p^{(v_p(4a)+1)s}}\right).
  \end{equation*}
  Similarly, we have $G_1(s)=H_2(s)L(s,\chi)$ with
  \begin{equation}\label{eq:euler_factors_H2}
    H_2(s)\coloneqq \prod_{\p \nmid 2a}
    \left(1+\frac{B_1(\p)}{\N\p^s}\right)\left(1-\frac{\legendre{a}{\p}}{\N\p^s}\right)
    = \prod_\p \left(1+O\left(\frac{1}{\N\p^{s+1}}\right)+O\left(\frac{1}{\N\p^{2s}}\right)\right)
  \end{equation}
  by (\ref{eq:B_1_bound}).  In total,
  \begin{equation*}
    F(s) = G(s)\zeta_K(s) = c_0G_0(s)\zeta_K(s) = c_0H_1(s)H_2(s)L(s,\chi)\zeta_K(s).
  \end{equation*}

  Now we discuss the analytic properties of these functions. By
  \cite[Satz~LXI, LXX]{MR1544310},
  $\zeta_K(s)$ has a simple pole with residue $\rho_K$ at
  $s=1$ and is otherwise analytic; its Dirichlet series
  converges absolutely for $\Re(s)>1$. For
  $L(s,\chi)$, we refer to Lemma~\ref{lem:L-function}.
  By (\ref{eq:B_1_bound}), we have $1+\frac{B_1(\p)}{\N\p^s} \ne 0$ for
  $\Re(s) > 0$, hence $H_1(s)$ is analytic with absolutely convergent
  Dirichlet series for $\Re(s) > 0$.  For $\Re(s) > 1/2$,
  because of (\ref{eq:euler_factors_H2}),
  the function $H_2(s)$ is analytic and its Dirichlet series converges
  absolutely.

  In total, this shows that $G(s)$ is analytic for $\Re(s)>1/2$. Its Dirichlet
  series and its Euler product converge absolutely for $\Re(s)>1$.
  Since the Euler products of $L(s,\chi)$, $H_1(s)$, $H_2(s)$ converge for
  $s=1$ to $L(1,\chi)$, $H_1(1)$, $H_2(1)$, respectively (by
  Lemma~\ref{lem:L-function}\ref{it:conditional_convergence} and absolute convergence, respectively), the same
  is true for $G_0(s)$. Multiplying by the convergent product $c_0$ shows that
  \begin{equation}\label{eq:euler_product_G(1)}
    G(1)=c_0G_0(1) =
    \prod_\p A_\p(0)\left(1+\frac{B_0(\p)}{\N\p^s}+\frac{B_0(\p^2)}{\N\p^{2s}}+\dots\right),
  \end{equation}
  where the right hand side must be convergent. A computation reveals that
  each $\p$-adic factor of (\ref{eq:euler_product_G(1)}) agrees with
  $\vartheta_{2,\p}$. Therefore, $G(1)=\vartheta_2$. Using
  Lemma~\ref{lem:eta_properties}, we check that
  \begin{equation}\label{eq:r_a(p)-bound}
    r_a(\p) \ge -1
  \end{equation}
  for all $\p$. Therefore, $\vartheta_{2,\p} \ne 0$ for all $\p$, and hence
  $G(1) \ne 0$.
  
  Therefore, $F(s)$ has a simple pole with residue $\rho_K\vartheta_2$ at
  $s=1$ and is otherwise analytic for
  $\Re(s)>1/2$. Its Dirichlet series is
  absolutely convergent for $\Re(s)>1$.

  Next, we discuss the vertical growth of our functions. We define
  $C\coloneqq 1+\epsilon$. By \cite[Satz~LXX]{MR1544310},
  $\zeta_K(s) \ll_\epsilon |\Im(s)|^{(\frac 1 2+2\epsilon)d}$ for
  $\Re(s) \ge -2\epsilon$ and $|\Im(s)| \ge 2$; by Lemma~\ref{lem:L-function}\ref{it:vertical_growth},
  the same bound holds for $L(s,\chi)$. By absolute convergence,
  $|\zeta_K(s)|$ and $|L(s,\chi)|$ are $\ll_\epsilon 1$ for
  $\Re(s) \ge 1+2\epsilon$. Define $k_\epsilon(\sigma) \coloneqq  (1+2\epsilon)d-d\sigma$. By
  Phragmen-Lindel\"of (as in \cite[Remark to Theorem~II.1.20]{MR3363366}), in
  the domain $-2\epsilon \le \Re(s) \le 1+2\epsilon$ with $|\Im(s)| \ge 2$, we
  have $\zeta_K(s)L(s,\chi) \ll_{a,\Re(s),\epsilon} |\Im(s)|^{k_\epsilon(\Re(s))}$
  for each $\Re(s)$ and
  $\zeta_K(s)L(s,\chi) \ll_{a,\epsilon} |\Im(s)|^{k_\epsilon(\Re(s))+\epsilon}$
  uniformly. The Euler factors of $H_1(s)$ are $1$ for $\p \nmid 2a\bfr$ and
  $1+O(\frac{1}{\N\p^s})$ for $\p \mid \bfr$ by (\ref{eq:B_1_bound}), hence
  $H_1(s) \ll_{a,\epsilon} (1+\epsilon)^{\omega_K(\bfr)}$ for
  $\Re(s) \ge \epsilon$. By absolute convergence and since $H_2(s)$ is
  independent of $\bfr$ by definition, $H_2(s) \ll_{a,\epsilon} 1$ for
  $\Re(s) \ge 1/2+\epsilon$. In total, for
  $1/2+\epsilon \le \Re(s) \le 1+2\epsilon$ and $|\Im(s)| \ge 2$, we have
  $F(s) \ll_{a,\Re(s),\epsilon} C^{\omega_K(\bfr)}|\Im(s)|^{k_\epsilon(\Re(s))}$ for
  each $\Re(s)$ and
  $F(s) \ll_{a,\epsilon} C^{\omega_K(\bfr)}|\Im(s)|^{k_\epsilon(\Re(s))+\epsilon}$
  uniformly.
  
  We apply Perron's formula as in \cite[Lemma~1.4.2]{zbMATH00807194} with
  $c=1+2\epsilon$ and $T=\max\{x^{1/(2d)},2\}$.  We have
  \begin{equation*}
    a_n \coloneqq  \sum_{\N\qfr = n} \vartheta_1(\qfr)
    \ll_a \sum_{\N\qfr = n} 2^{\omega_K(\qfr)} \ll_\epsilon n^\epsilon
  \end{equation*}
  since there are $\ll n^{\epsilon/2}$ ideals of norm $n$, and each
  satisfies $2^{\omega_K(\qfr)} \ll_\epsilon \N\qfr^{\epsilon/2}$.
 Therefore, 
  \begin{equation*}
    \sum_{\N\qfr \le x} \vartheta_1(\qfr) = \frac{1}{2\pi i}
    \int_{c-iT}^{c+iT} F(s) \frac{x^s}{s} \dd s + R
  \end{equation*}
  with the error term
  \begin{align*}
    R & \ll \frac{x^c}{T} \sum_{n=1}^\infty \frac{|a_n|}{n^c} +
    \left(1+\frac{x \log x}{T}\right)\max_{3x/4\le n \le 5x/4}\{|a_n|\}\\
    &\ll_{a,\epsilon} \frac{x^{1+2\epsilon}}{T} \sum_{n=1}^\infty
    \frac{n^\epsilon}{n^{1+2\epsilon}} + \frac{x \log x}{T} (5x/4)^\epsilon\\
    &\ll_\epsilon x^{1-1/(2d)+2\epsilon}.
  \end{align*}
  Let $b = 1/2+\epsilon$. By Cauchy's residue theorem for the rectangle with
  vertices $b\pm iT, c\pm iT$, the main term is
  \begin{equation*}
    \Res_{s=1}\left(F(s)\frac{x^s}{s}\right) +
    \frac{1}{2\pi i}\left(
      -\int_{b-iT}^{c-iT}+\int_{b-iT}^{b+iT}+\int_{b+iT}^{c+iT}\right)
    F(s) \frac{x^s}{s} \dd s.
  \end{equation*}
  Since $\Res_{s=1}F(s)=\rho_K\vartheta_2$, the residue is
  $\rho_K\vartheta_2x$.  The first and third integral are (using
  $1/|s|\le 1/T$)
  \begin{align*}
    &\ll_\epsilon \int_b^c C^{\omega_K(\bfr)}T^{(1+2\epsilon)d-d\sigma+\epsilon} \frac{x^\sigma}{T} \dd
    \sigma
    \ll C^{\omega_K(\bfr)}T^{(1+2\epsilon)d+\epsilon-1} \int_b^{c} (x/T^d)^\sigma \dd \sigma\\
    &\ll C^{\omega_K(\bfr)}T^{(1+2\epsilon)d+\epsilon-1} \frac{(x/T^d)^c}{\log(x/T^d)} \ll
    C^{\omega_K(\bfr)}\frac{x^c}{T^{1-\epsilon}} \ll C^{\omega_K(\bfr)}x^{1-1/(2d)+3\epsilon}.
  \end{align*}
  For the second integral, we use $|F(s)| \le |\Im(s)|^{k(b)}$ and
  $|s| \ge |\Im(s)|$ (for $|\Im(s)| \ge 2$) and $|F(s)| \ll 1$ and $|s|\ge b$
  (for $|\Im(s)| \le 2$)
  to deduce
  \begin{align*}
    &\ll_\epsilon \int_0^2 C^{\omega_K(\bfr)}\frac{x^b}{b} \dd \tau + \int_2^T
      C^{\omega_K(\bfr)}\tau^{(1/2+\epsilon)d} \frac{x^b}{\tau} \dd \tau\\
    &\ll C^{\omega_K(\bfr)}x^{1/2+\epsilon}T^{(1/2+\epsilon)d}=C^{\omega_K(\bfr)}x^{3/4+3\epsilon/2}.
  \end{align*}
  In total, this proves the result.
\end{proof}

\subsection{Completion of the proof}\label{sec:remaining_summations_completion}

\begin{proof}[Proof of Theorem~\ref{thm:number_fields}]
  Our starting point is Proposition~\ref{prop:first_summation}. We argue similarly as
  in \cite[\S 7]{MR3269462}, but $\vartheta_1$ is more complicated.

  First, we sum over $\afr_1$.  For $t_1,\dots,t_4 \in \RR_{\ge 1}$, let
  \begin{equation*}
    V_1(t_1,\dots,t_4;B)\coloneqq 
    \begin{cases}
      \frac{B}{t_1\cdots t_4}, & t_1^2t_2t_3^2 \le B,\ t_1^{-1}t_2^4t_3^2t_4^6
      \le B,\\
      0, & \text{else.}
    \end{cases}
  \end{equation*}
  Since the two inequalities in the definition of $V_1$ (which correspond to
  the height conditions (\ref{eq:height_1234_ideal})) imply
  $t_1t_2^2t_3^2t_4^2 \le B$, we have $V_1(t_1,\dots,t_4;B)=0$ unless
  $t_1,\dots,t_4 \le B$.  We claim that
  \begin{align*}
    \sum_{\afrb \in \Is_K^4}
    &\vartheta_1(\afrb)V_1(\N\afr_1,\dots,\N\afr_4;B)\\
    ={}&
    \rho_K \sum_{\afr_2,\afr_3,\afr_4 \in \Is_K^3} \vartheta_2(\afr_2,\afr_3,\afr_4)
    \int_1^\infty V_1(t_1,\N\afr_2,\N\afr_3,\N\afr_4;B) \dd t_1 + O(B(\log B)^3).
  \end{align*}
  Indeed, by Lemma~\ref{lem:summation_theta_1}, we can apply
  \cite[Lemma~2.10]{MR3269462} with $m=1$,
  $c_1=(1+\epsilon)^{\omega_K(\afr_2\afr_3\afr_4)}$, $b_1=1-1/(2d)+\epsilon$,
  $k_1=0$, $c_g = B/\N(\afr_2\afr_3\afr_4)$, $a=-1$ for the summation over
  $\afr_1$. By \cite[Lemmas~2.9, 2.4, 2.10]{MR3269462}, the total error term
  is
  \begin{equation*}
    \ll \sum_{\N\afr_2,\N\afr_3,\N\afr_4 \le B}
    \frac{(1+\epsilon)^{\omega_K(\afr_2\afr_3\afr_4)}B}{\N(\afr_2\afr_3\afr_4)}
    \ll B(\log B)^{3+3\epsilon}.
  \end{equation*}

  For the summations over $\afr_2,\afr_3,\afr_4$, let
  \begin{equation*}
    \theta_{2,\p}(\vv) = \frac{\vartheta_{2,\p}(\vv)}{\vartheta_{2,\p}(0,0,0)} \in \RRnn
  \end{equation*}
  and
  \begin{equation*}
    \theta_2(\afr_2,\afr_3,\afr_4) = \prod_p \theta_{2,\p}(v_\p(\afr_2),v_\p(\afr_3),v_\p(\afr_4)),
  \end{equation*}
  hence 
  \begin{equation*}
    \vartheta_2(\afr_2,\afr_3,\afr_4) =
    \theta_2(\afr_2,\afr_3,\afr_4)\cdot\prod_\p\vartheta_{2,\p}(0,0,0).
  \end{equation*}
  
  Then $\theta_2 \in \Theta_3'(C)$ \cite[Definition~2.6]{MR3269462} for
  $C = \max(\{\N\p : \p \mid 2a\} \cup \{5\})$. Indeed,
  $\theta_2(\afr_2,\afr_3,\afr_4) \in [0,1]$ since
  \begin{equation*}
    0 < \left(1-\frac{1}{\N\p}\right)^4 < \left(1-\frac{1}{\N\p}\right)^3 < \left(1-\frac{1}{\N\p}\right)^2\left(1+\frac{2+r_a(\p)}{\N\p}\right)
  \end{equation*}
  is easily checked using~(\ref{eq:r_a(p)-bound}). By definition,
  $\theta_2(0,0,0)=1$. If $|\supp(\vv)|=1$, then
  $\theta_{2,\p}(\vv) \ge 1-\frac{C}{\N\p}$ holds for $\p \nmid 2a$ using
  $|r_a(\p)|=1$ and $C \ge 5$, and it holds for $\p \mid 2a$ since
  $\theta_{2,\p}(\vv) \ge 0$.

  Therefore, we can apply \cite[Proposition~7.2]{MR3269462} inductively (as in
  the proof of \cite[Proposition~7.3]{MR3269462}; note that our definition (\ref{eq:rho_K_def}) of
  $\rho_K$ differs from the one in \cite[\S 1.4]{MR3269462} by the factor $h_K$) to deduce that
  \begin{align*}
    &\sum_{\afr_2,\afr_3,\afr_4} \vartheta_2(\afr_2,\afr_3,\afr_4)
    \int_1^\infty V_1(t_1,\N\afr_2,\N\afr_3,\N\afr_4;B) \dd t_1\\
    &= \rho_K^3\vartheta_3 \int_{t_1,\dots,t_4 \ge 1} V(t_1,\dots,t_4;B) \dd
      t_1\cdots \dd t_4 + O(B(\log B)^3 \log \log B)
  \end{align*}
  where
  \begin{equation*}
    \vartheta_3\coloneqq \As(\As(\As(\vartheta_2(\afr_2,\afr_3,\afr_4),\afr_2),\afr_3),\afr_4)
  \end{equation*}
  is the average value of $\vartheta_2(\afr_2,\afr_3,\afr_4)$ when summed over
  $\afr_2,\afr_3,\afr_4$ (as defined before \cite[Proposition~2.3]{MR3269462}).
  
  Therefore,
  \begin{equation*}
    \sum_{\afrb \in \Is_K^4:(\ref{eq:height_1234_ideal})}
    \frac{\vartheta_1(\afr)B}{\N(\afr_1\afr_2\afr_3\afr_4)} =
    \rho_K^4
    \vartheta_3 V_0(B)+O(B(\log B)^3(\log \log B)),
  \end{equation*}
  with
  \begin{equation*}
    V_0(B)\coloneqq \frac{1}{3} \ints{t_1,\dots, t_4 \ge 1\\
        t_1^2t_2t_3^2 \le B\\ t_1^{-1}t_2^4t_3^2t_4^6\le B} \frac{B}{t_1t_2t_3t_4}
    \dd t_1 \dd t_2 \dd t_3 \dd t_4.
  \end{equation*}

  The classes of $E_1, \dots, E_5$ form a basis of $\Pic(\tS)$. Together with
  $[E_6] = [E_1]-[E_2]-2[E_4]+[E_5]$, they generate the effective cone. We
  have $-K_\tS = 2[E_1]+[E_2]+2[E_3]+3[E_5]$. As in
  \cite[Lemma~8.1]{MR3269462}, 
  \begin{equation*}
    V_0(B)=3\alpha(\tS)B(\log B)^4.
  \end{equation*}
  Here, by \cite[Corollary~7.5]{MR2377367},
  \begin{equation*}
    \alpha(\tS) = \frac{\alpha(S_0)}{|W(R)|} = \frac{1}{1728} = \alpha_{S}
  \end{equation*}
  as in Theorem~\ref{thm:number_fields}, where $S_0$ is a nonsplit smooth
  quartic del Pezzo surface with $\Pic(S_0)$ of rank $5$ (with
  $\alpha(S_0) = \frac{1}{36}$ by \cite[Theorem~4.1]{MR3143700}) and
  $|W(R)| = 2!\cdot 4!$ is the order of the Weyl group of the root system
  $\Athree+\Aone$ (since all $(-2)$-curves on $\tS_\Kbar$ are already defined
  over $K$).

  We check that $\vartheta_3 = \prod_\p \omega_\p(S)$ with $\omega_\p(S)$
  as in Theorem~\ref{thm:number_fields} using \cite[Lemma~2.8, formula
  (2.2)]{MR3269462}, which we may apply to $\theta_2$ and therefore also
  directly to $\vartheta_2$.

  By the discussion in the
  following Section~\ref{sec:constant}, our asymptotic formula agrees
  with the conjectures of Manin and Peyre.
\end{proof}

\section{The leading constant}\label{sec:constant}

In this section, we check that the leading constant in
Theorem~\ref{thm:number_fields} agrees with \cite[Formule
empirique~5.1]{MR2019019}.  Recall that $\phi\colon \tS \to S$ is a minimal
desingularization.  Then $\tS$ is an almost Fano variety in the sense of
\cite[Definition~3.1]{MR2019019} by the description in
\cite[\S3.4]{math.AG/0604194} and an application of Kawamata--Viehweg
vanishing \cite[Corollary~0.3]{MR667459}. Moreover, $\tS$ satisfies
\cite[Hypoth\`ese~3.3]{MR2019019} because it is rational.

\subsection{Convergence factors}\label{sec:convergence_factors}

As in \cite[\S 13]{MR3552013}, we construct an adelic metric on the
anticanonical line bundle $\omega_\tS^{-1}$ such that the corresponding height
function on $\tS(K)$ is $H \circ \phi$.  We denote by $\omega_{H,v}$ the local
measures, and by $\tau_H$ the Tamagawa measure corresponding to the adelic
metric (see \cite[\S 4]{MR2019019}). Then the leading constant is expected to
be
\begin{equation*}
  c_{S,H} = \alpha(\tS) \beta(\tS) \tau_H(\tS)
\end{equation*}
where $\alpha(\tS) = \alpha_S$ as discussed in
Section~\ref{sec:remaining_summations}, $\beta(\tS)=1$ as $\Pic(\tS_{\Kbar})$
is a permutation $\Gal(\Kbar/K)$-module and
\begin{equation*}
  \tau_H(\tS) = \lim_{s\to 1}(s-1)^5L_{S^0}(s,\Pic(\tS_{\Kbar})) \cdot
  \frac{1}{|\Delta_K|} \cdot \prod_{v \in \Omega_K} \lambda_v^{-1}\omega_{H,v}(\tS(K_v)),
\end{equation*}
for a finite set $S^0$ of finite places of $K$,
as defined in \cite[Notations~4.5]{MR2019019}.

\begin{lemma}
  We have
  \begin{multline*}
    \lim_{s \to 1} (s-1)^5 L_{S^0}(s,\Pic(\tS_\Kbar))\prod_{v \in \Omega_K}
    \lambda_v^{-1}
    \omega_{H,v}(\tS(K_v))\\
    = \rho_K^5 \left(\prod_{v\in\Omega_{\infty}}\omega_{H,v}(\tS(K_v))\right)
    \prod_{\p\in\Omega_f}
    \left(\left(1-\frac{1}{\N\p^s}\right)^5\omega_{H,\p}(\tS(K_\p))\right).
  \end{multline*}
\end{lemma}

\begin{proof}
  We follow the strategy of \cite[Proposition~5.1]{MR1681100}.  
 We use the notation $S^c\coloneqq\Omega_f\smallsetminus S^0$.  
  With
  $\Lambda\coloneqq \Pic(\tS_\Kbar)$, by definition,
  \begin{equation*}
    L_{S^0}(s,\Lambda) = \prod_{\p \in S^c} L_\p(s, \Lambda),
    \qquad
    \lambda_v=
    \begin{cases}
      L_v(1,\Lambda), & v \in S^c,\\
      1, & v \in S^0\cup\Omega_\infty,
    \end{cases}
  \end{equation*}
  with
  \begin{equation*}
    L_\p(s,\Lambda) = \frac{1}{\det(1-|\FFF_\p|^{-s} \Fr_\p)},
  \end{equation*}
  where $\Fr_\p$ is the geometric Frobenius acting on
  $\Pic(\tSS \times_\ZZ \FFF_\p) \otimes_\ZZ \QQ$.
 
  Enlarging $S^0$ does not change the expression that we want to
  compute. Therefore, we may assume that $S^0$ contains $\{\p \in\Omega_f : \p \mid 2a\}$
  and that there is a model $\tSS$ of $\tS$ over the ring $\OO_{S^0}$ of
  $S^0$-integers such that
  \begin{equation*}
    \omega_{H,\p}(\tS(K_\p)) = \frac{|\tSS(\FFF_\p)|}{|\FFF_\p|^2}
  \end{equation*}
  for all $\p \in S^c$ by \cite[Remarque~4.4]{MR2019019}.  By a result of Weil
  \cite[Theorem~23.1]{MR833513}, for $\p \in S^c$,
  \begin{equation*}
    |\tSS(\FFF_\p)| = |\FFF_\p|^2+|\FFF_\p| \Tr(\Fr_\p)+1.
  \end{equation*}
 In Lemma~\ref{lem:tamagawa_p_numberfield} we show that for finite places $\p \nmid 2a$,
  \begin{equation*}
    \omega_{H,\p}(\tS(K_\p)) = 1+\frac{5+\legendre{a}{\p}}{\N\p} + \frac{1}{\N\p^2}.
  \end{equation*}
  Since $|\FFF_\p| = \N\p$, we deduce, for $\p \in S^c$, that
  \begin{equation*}
    \Tr(\Fr_\p) = 5+\legendre{a}{\p}.
  \end{equation*}
  
  We have
  $\det(1-T\cdot f) = 1-\Tr(f)\cdot T + \dots + (-1)^n\det(f)\cdot T^n$ in the
  polynomial ring $\QQ[T]$ for an endomorphism $f$ of an $n$-dimensional
  $\QQ$-vector space. Therefore,
  \begin{equation*}
    L_\p(s,\Lambda) = 1+\frac{5+\legendre{a}{\p}}{\N\p^s}+O\left(\frac{1}{\N\p^{2s}}\right)
    = \left(\frac{1}{1-\frac{1}{\N\p^s}}\right)^5 \cdot
    \frac{1}{1-\frac{\chi(\p)}{\N\p^s}} \cdot  G_\p(s)
  \end{equation*}
  for some function
  \begin{equation*}
    G_\p(s) = 1+O\left(\frac{1}{\N\p^{2s}}\right).
  \end{equation*}
  With $\chi$ as in~(\ref{eq:def_chi}), let
  \begin{equation*}
    L_{S^0}(s,\chi) \coloneqq  \prod_{\p \in S^c}
    \frac{1}{1-\frac{\chi(\p)}{\N\p^s}},\qquad 
    \zeta_{K,S^0}(s) \coloneqq  \prod_{\p
      \in S^c} \frac{1}{1-\frac{1}{\N\p^s}},\qquad 
      G(s) \coloneqq  \prod_{\p \in
      S^c} G_\p(s).
  \end{equation*}
  Here, $G(s)$ is analytic for $\Re(s) > 1/2$, while $\zeta_{K,S^0}$ has a
  simple pole with residue
  \begin{equation*}
    \rho_{K,S^0}\coloneqq \rho_K\prod_{\p \in S^0\cap\Omega_f}\left(1-\frac{1}{\N\p}\right)
  \end{equation*}
  at $s=1$ and $L_{S^0}(1,\chi)$ converges and is nonzero by
  Lemma~\ref{lem:L-function}\ref{it:value_at_1}.  Therefore,
  \begin{equation*}
    L_{S^0}(s,\Lambda) = \zeta_{K,S^0}(s)^5L_{S^0}(s,\chi)G(s),
  \end{equation*}
  with
  \begin{equation*}
    \lim_{s \to 1} (s-1)^5 L_{S^0}(s,\Lambda)
    = \rho_{K,S^0}^5 L_{S^0}(1,\chi) G(1)
    = \rho_K^5 L_{S^0}(1,\chi) G(1)
    \!\!\prod_{\p \in S^0\cap\Omega_f}\!\!\left(1-\frac{1}{\N\p}\right)^5.
  \end{equation*}
  We compute
  \begin{align*}
    &\prod_{v\in\Omega_f} \lambda_v^{-1} \omega_{H,v}(\tS(K_v))\\
    ={}& L_{S^0}(1,\chi)^{-1}G(1)^{-1}
    \prod_{\p\in\Omega_f}
      \left(\left(1-\frac{1}{\N\p}\right)^5\omega_{H,\p}(\tS(K_\p))\right)
      \prod_{\p \in S^0\cap\Omega_f} \left(\frac{1}{1-\frac{1}{\N\p}}\right)^{-5}
  \end{align*}
  since $L_{S^0}(1,\chi) \ne 0$ by Lemma~\ref{lem:L-function}\ref{it:value_at_1}.
\end{proof}

\subsection{Local densities}\label{sec:local_densities}

We compute the local densities $\omega_{H,v}(\tS(K_v))$.  The birational projection $S \dashrightarrow \PP^2$,
$(x_0:\dots:x_4) \mapsto (x_1:x_3:x_4)$ induces the chart
\begin{align*}
  \AAA^2\smallsetminus\{x_3=0\} &\to S,\\
  (x_1,x_3) &\mapsto (ax_3-x_1^2:x_1:x_3^{-1}:x_3:1).
\end{align*}
With $f_1=x_0x_4+x_1^2-ax_3^2$, $f_2=x_2x_3-x_4^2$, we have
\begin{equation*}
  \det
  \begin{pmatrix}
    \dd f_1/\dd x_0 & \dd f_1/\dd x_2 \\
    \dd f_2/\dd x_0 & \dd f_2/\dd x_2
  \end{pmatrix}
  =x_3x_4,
\end{equation*}
hence with $x_4=1$, we deduce analogously to \cite[\S 13]{MR3552013} for every
$v \in \Omega_K$ that
\begin{equation}\label{eq:density_integral}
  \omega_{H,v}(\tS(K_v)) = \int_{K_v^2}
  \frac{\dd x_1 \dd x_3}{|x_3|_v\cdot \max\{|ax_3^2-x_1^2|_v,|x_1|_v,|x_3^{-1}|_v,|x_3|_v,1\}}.
\end{equation}

\begin{lemma}\label{lemma:omega_infty_archimedean}
  For every archimedean place $v$ of $K$, the constant $\omega_v(S)$ defined
  in Theorem~\ref{thm:number_fields} agrees with the $v$-adic density
  $\omega_{H,v}(\tS(K_v))$.
\end{lemma}

\begin{proof}
  This is analogous to the end of \cite[\S 13]{MR3552013}.
\end{proof}

For a prime ideal $\p$ of $\OO_K$, let $\mu_\p$ be the Haar measure on $K_\p$, normalized such that $\mu_\p(\{x \in K_\p : v_\p(x) \ge 0\})=1$.

\begin{lemma}\label{lem:measure_squares}
  Let $\p$ be a prime ideal in $\OO_K$. Assume that $v_\p(a)$ is
  even. Let $x_3 \in K_\p^\times$, and let $k \ge 0$. Then
  \begin{equation*}
    \mu_\p(\{x_1 \in K_\p : v_\p(a(x_3/x_1)^2-1) \ge k\})
    = \frac{\eta(\p^{v_\p(a)+k};a)}{\N\p^{v_\p(ax_3^2)/2+k}}.
  \end{equation*}
\end{lemma}

\begin{proof}
  Let $\gamma=2\gamma'=v_\p(a)\ge 0$. First, we consider the case
  $x_3=1$.
  The condition $v_\p(a-x_1^2) \ge 2\gamma'+k$ holds if and
  only if $v_\p(x_1)=\gamma'$ and $\congr{x_1}{\rho}{\p^{\gamma'+k}}$
  for some $\rho \in \OO_K$ satisfying
  $\rho\OO_K + \p^{\gamma'+k} = \p^{\gamma'}$ and
  $\congr{\rho^2}{a}{\p^{2\gamma'+k}}$. Since
  $\eta(\p^{2\gamma'+k};a)$ is defined as the cardinality of the set
  of such $\rho \pmod{\p^{\gamma'+k}}$, our claim holds in the first
  case.

  In the general case, the transformation $x_1 \mapsto x_1/x_3$
  reduces to the first case.
\end{proof}

\begin{lemma}\label{lem:sum_kpk}
  For every $q \in \RR_{>1}$ and $n \in \ZZnn$, we have
  \begin{equation*}
    \sum_{k > n}\frac{k}{q^k}
    =\left(1-\frac 1 q\right)^{-2}\left(\frac{n+1}{q^{n+1}}-\frac{n}{q^{n+2}}\right).
  \end{equation*}
\end{lemma}

\begin{proof}
  We observe that
  \begin{equation*}
    \frac{n+1}{q^{n+1}}+\frac{n+2}{q^{n+2}}+\frac{n+3}{q^{n+3}}+\dots
    =\left(1+\frac{1}{q}+\frac{1}{q^2}+\dots\right)^2
    \left(\frac{n+1}{q^{n+1}}-\frac{n}{q^{n+2}}\right).\qedhere
  \end{equation*}
\end{proof}

\begin{lemma}\label{lem:tamagawa_p_numberfield}
  For each finite place $\p$ of $K$, we have
  \begin{equation*}
    \left(1-\frac{1}{\N\p}\right)^5\omega_{H,\p}(\tS(K_\p)) = \omega_\p(S),
  \end{equation*}
  with $\omega_\p(S)$ as in
  Theorem~\ref{thm:number_fields}.
\end{lemma}

\begin{proof}
  Let $q\coloneqq \N\p$. Let $\alpha=v_\p(x_1)$, $\beta=v_\p(x_3)$, and
  $\gamma=v_\p(a)$. In the case $2\alpha=2\beta+\gamma$, let
  $\kappa\coloneqq v_\p(a(x_3/x_1)^2-1) \in \ZZnn$. Then
  \begin{equation*}
    v_\p(ax_3^2-x_1^2) =
    \begin{cases}
      2\alpha, & 2\alpha<2\beta+\gamma,\\
      2\beta+\gamma, & 2\alpha>2\beta+\gamma,\\
      2\beta+\gamma+\kappa, & 2\alpha=2\beta+\gamma.
    \end{cases}
  \end{equation*}
  Now it is straightforward to compute the maximum $M$ in
  (\ref{eq:density_integral}) as
  \begin{equation*}
    M=
    \begin{cases}
      q^\beta, & \beta \ge 0,\ \alpha\ge -\beta/2,\\
      q^{-2\alpha}, & \alpha \le 0,\ \max\{\alpha-\gamma/2,2\alpha\} < \beta
      \le -2\alpha,\\
      q^{-\beta}, & -\gamma \le \beta \le 0,\ \alpha \ge \beta/2,\\
      q^{-(2\beta+\gamma)}, & \beta<-\gamma,\ \alpha>\beta+\gamma/2,\\
      q^{-(2\beta+\gamma+\kappa)},
      & \beta<-\gamma,\ \alpha=\beta+\gamma/2,\ 0 \le \kappa < -\beta-\gamma,\\
      q^{-\beta}, & \beta<-\gamma,\ \alpha=\beta+\gamma/2,\ \kappa \ge -\beta-\gamma,\\
    \end{cases}
  \end{equation*}
  The integration domain
  $K_\p^2=\{(x_1,x_3) \in K_\p^2 : v_\p(x_1)=\alpha,\ v_\p(x_3)=\beta\}$ in
  (\ref{eq:density_integral}) is the disjoint union of the sets
  \begin{align*}
    V_1&=\{(x_1,x_3) \in K_\p^2 : \beta > 0,\ \alpha \ge -\beta/2,\ (\alpha,\beta)\ne(0,0)\},\\
    V_2&=\{(x_1,x_3) \in K_\p^2 : -\gamma/2 \le \alpha<0,\ 2\alpha < \beta < -2\alpha\},\\
    V_3&=\{(x_1,x_3) \in K_\p^2 : \alpha<-\gamma/2,\ \alpha-\gamma/2 <\beta<-2\alpha\},\\
    V_4&=\{(x_1,x_3) \in K_\p^2 : -\gamma \le \beta \le 0,\ \alpha \ge \beta/2\},\\
    V_5&=\{(x_1,x_3) \in K_\p^2 : \beta<-\gamma,\ \alpha>\beta+\gamma/2\},\\
    V_6&=\{(x_1,x_3) \in K_\p^2 : 2\alpha=2\beta+\gamma,\ \beta<-\gamma\}.
  \end{align*}
  Therefore,
  \begin{equation*}
    \omega_{H,\p}(\tS(K_\p))=\omega_{\p,1}+\omega_{\p,2}+\omega_{\p,3}
    +\omega_{\p,4}+\omega_{\p,5}+\omega_{\p,6}
  \end{equation*}
  where $\omega_{\p,i}$ is (\ref{eq:density_integral}) integrated over $V_i$
  instead of $K_\p^2$.
  
  On $V_1$, we have $|x_3|_\p\cdot M=q^{-\beta}\cdot q^\beta=1$. Writing
  $\beta=2\beta'-1$ or $\beta=2\beta'$ with $\beta' \in \ZZp$, we have
  \begin{align*}
    \omega_{\p,1}
    &= \sum_{\beta' > 0} \ints{\alpha \ge -\beta'+1/2\\\beta = 2\beta'-1} \dd x_1\, \dd x_3
    + \ints{\alpha \ge -\beta'\\\beta = 2\beta'} \dd x_1\, \dd x_3\\
    &=\sum_{\beta' > 0} q^{\beta'-1}\cdot \left(1-\frac 1 q\right)q^{-2\beta'+1}
    +q^{\beta'}\cdot \left(1-\frac 1 q\right)q^{-2\beta'}\\
    &=\left(1-\frac 1 q\right)\sum_{\beta'>0}2q^{-\beta'}= \frac 2 q.
  \end{align*}  
  On $V_2$, we have $|x_3|_\p\cdot M=q^{-\beta}\cdot q^{-2\alpha}$. Writing
  $\alpha'=-\alpha$ and $\gamma'=\floor{\gamma/2}$ and using
  Lemma~\ref{lem:sum_kpk},
  \begin{align*}
    \omega_{\p,2}
    &= \ints{-\gamma/2 \le \alpha < 0\\2\alpha < \beta < -2\alpha}
    \frac{\dd x_1\,\dd x_3}{q^{-2\alpha-\beta}}
    =\sum_{-\gamma/2 \le \alpha < 0} \sum_{2\alpha < \beta < -2\alpha}
    q^{2\alpha+\beta} q^{-\alpha} q^{-\beta} \left(1-\frac 1 q\right)^2\\
    &= \left(1-\frac 1 q\right)^2 \sum_{0 < \alpha' \le \floor{\gamma/2}}
      \frac{4\alpha'-1}{q^{\alpha'}}=\frac{3}{q}+\frac{1}{q^2}
      -\frac{4\gamma'+3}{q^{\gamma'+1}} + \frac{4\gamma'-1}{q^{\gamma'+2}}.
  \end{align*}
  On $V_3$, we have $|x_3|_\p\cdot M=q^{-\beta}\cdot q^{-2\alpha}$. Writing
  $\alpha'=-\alpha$ and using Lemma~\ref{lem:sum_kpk},
  \begin{align*}
    \omega_{\p,3}
    &= \ints{\alpha<-\gamma/2\\\alpha-\gamma/2<\beta<-2\alpha}
    \frac{\dd x_1\,\dd x_3}{q^{-2\alpha-\beta}}
    =\left(1-\frac 1 q\right)^2\sum_{\alpha'>\floor{\gamma/2}}
    \frac{3\alpha'+\ceil{\gamma/2}-1}{q^{\alpha'}}\\
    &=\begin{cases}
      \frac{4\gamma'+2}{q^{\gamma'+1}}-\frac{4\gamma'-1}{q^{\gamma'+2}}, & \gamma=2\gamma'\text{ even},\\
      \frac{4\gamma'+3}{q^{\gamma'+1}}-\frac{4\gamma'}{q^{\gamma'+2}}, & \gamma=2\gamma'+1\text{ odd}.
    \end{cases}
  \end{align*}
  On $V_4$, we have $|x_3|_\p\cdot M=q^{-\beta}\cdot
  q^{-\beta}$. Writing $\beta'=-\beta$, we have
  \begin{align*}
    \omega_{\p,4}
    &= \ints{-\gamma \le \beta \le 0\\\alpha \ge \beta/2} \frac{\dd x_1\,\dd x_3}{q^{-2\beta}}\
    =\sum_{\beta'=0}^\gamma q^{-\ceil{\beta'/2}}\left(1-\frac 1 q\right)\\
    &=1+\frac 1 q-\frac{1}{q^{\floor{\gamma/2}+1}}-\frac{1}{q^{\floor{(\gamma-1)/2}+2}}\\
  \end{align*}
  On $V_5$, we have $|x_3|_\p \cdot M = q^{-\beta} \cdot
  q^{-(2\beta+\gamma)}$. Writing $\beta'=-\beta$, we have
  \begin{align*}
    \omega_{\p,5}
    &= \ints{\alpha > \beta+\gamma/2\\\beta<-\gamma} \frac{\dd x_1\,\dd x_3}{q^{-3\beta-\gamma}}
    =\sum_{\beta<-\gamma} q^{3\beta+\gamma} \left(1-\frac 1 q\right)
    q^{-\beta}
    \cdot q^{-\floor{\beta+\gamma/2}-1}\\
    &=\left(1-\frac 1 q\right)q^{\gamma-\floor{\gamma/2}-1}\sum_{\beta'>\gamma} q^{-\beta'}
      =q^{-\floor{\gamma/2}-2}.
  \end{align*}
  So far, this shows that
  \begin{equation*}
    \omega_{\p,1}+\omega_{\p,2}+\omega_{\p,3}+\omega_{\p,4}+\omega_{\p,5} = 1+\frac 6 q+\frac{1}{q^2}-
    \begin{cases}
      \frac{3}{q^{\gamma'+1}}-\frac{1}{q^{\gamma'+2}}, & \gamma=2\gamma'\text{ even},\\
      \frac{1}{q^{\gamma'+1}}+\frac{1}{q^{\gamma'+2}}, & \gamma=2\gamma'+1\text{ odd}.      
    \end{cases}
  \end{equation*}
  For odd $\gamma$, since $V_6=\emptyset$ and $\omega_{\p,6}=0$, this is
  $\omega_{H,\p}(\tS(K_\p))$, which agrees with
  Theorem~\ref{thm:number_fields}.

  Now assume that $\gamma=2\gamma'=v_\p(a) \ge 0$ is even.
  By the definition of $\kappa$, 
  \begin{equation*}
    |x_3|_\p\cdot M=q^{-\beta}\cdot q^{-\min\{2\beta+\gamma+\kappa,\beta\}} = q^{-\beta}\cdot
    \begin{cases}
      q^{-(2\beta+\gamma+\kappa)}, & 0 \le \kappa <-\beta-\gamma,\\
      q^{-\beta}, & \kappa \ge -\beta-\gamma.
    \end{cases}
  \end{equation*}
  Therefore, using Lemma~\ref{lem:measure_squares} in the second step,
  writing $\beta' = -\beta$ in the third step, and exchanging orders
  of summation in the fourth step,
  \begin{align*}
    \omega_{\p,6}
    &= \int_{\beta < -\gamma} \left(\sum_{0\le k < -\beta-\gamma}
      \ints{\alpha=\beta+\gamma'\\\kappa = k} \frac{\dd x_1}{q^{-(3\beta+\gamma+\kappa)}}
    +\ints{\alpha=\beta+\gamma'\\\kappa \ge -\beta-\gamma} \frac{\dd x_1}{q^{-2\beta}}\right)\dd x_3\\
    &= \sum_{\beta < -\gamma} \left(1-\frac 1 q\right)q^{-\beta}
      \Bigg(q^{3\beta+\gamma}\left(\left(1-\frac 1 q\right) q^{-(\beta+\gamma')} - \frac{\eta(\p^{\gamma+1};a)}{q^{\beta+\gamma'+1}}\right)\\
    &+ \sum_{0 < k < -\beta-\gamma} q^{3\beta+\gamma+k}\left(\frac{\eta(\p^{\gamma+k};a)}{q^{\beta+\gamma'+k}}
      -\frac{\eta(\p^{\gamma+k+1};a)}{q^{\beta+\gamma'+k+1}}\right)+q^{2\beta}\frac{\eta(\p^{\gamma+(-\beta-\gamma)};a)}{q^{\beta+\gamma'+(-\beta-\gamma)}}\Bigg) \\
    &=\left(1-\frac 1 q\right)^2\sum_{\beta'>\gamma} q^{-\beta'+\gamma'}\left(1+\sum_{0 < k \le \beta'-\gamma} \eta(\p^{\gamma+k};a)\right)\\
    &=\frac{1}{q^{\gamma'+1}}-\frac{1}{q^{\gamma'+2}}
      +\left(1-\frac 1 q\right)^2 q^{\gamma'}\sum_{k=1}^\infty
      \left(\eta(\p^{\gamma+k};a)\sum_{\beta'=\gamma+k}^\infty
      \frac{1}{q^{\beta'}}\right)\\
    &=\frac{1}{q^{\gamma'+1}}-\frac{1}{q^{\gamma'+2}}+\left(1-\frac 1 q\right)\sum_{k=1}^\infty
      \frac{\eta(\p^{\gamma+k};a)}{q^{\gamma'+k}}.
  \end{align*}
  For even $\gamma$, we have in total
  \begin{equation*}
    \omega_{H,\p}(\tS(K_\p)) =  \left(1+\frac 6 q+\frac{1}{q^2}\right)+\frac{1}{q^{\gamma'}}
    \left(-\frac{2}{q}+\left(1-\frac 1 q\right)\sum_{k=1}^\infty
      \frac{\eta(\p^{\gamma+k};a)}{q^k}\right),
  \end{equation*}
  which also agrees with Theorem~\ref{thm:number_fields}.
\end{proof}

\bibliographystyle{alpha}

\bibliography{manin_dp4_a1a3ns_nf}

\end{document}